\documentclass[11pt]{amsart}
\usepackage{epsfig}
\usepackage{graphicx, subfigure}
\usepackage{amsmath, amssymb,mathabx,mathrsfs}
\usepackage{color}
\usepackage[toc,page]{appendix}
\usepackage{comment}
\usepackage{hyperref}

\newtheorem{thm}{Theorem}[section]
\newtheorem{cor}[thm]{Corollary}
\newtheorem{lem}[thm]{Lemma}
\newtheorem{prop}[thm]{Proposition}

\theoremstyle{definition}
\newtheorem{defn}[thm]{Definition}
\theoremstyle{remark}
\newtheorem{rem}[thm]{Remark}
\numberwithin{equation}{section}

\newcommand{\R}{\mathbb R}
\newcommand{\T}{\mathbb T}

\newcommand{\eps}{\varepsilon}

\newcommand{\C}{\mathcal{C}}

\newcommand{\X}{\mathcal{X}}

\newcommand{\st}{{\rm s}}           
\newcommand{\un}{{\rm u}}           
\newcommand{\sst}{{\rm ss}}           
\newcommand{\uun}{{\rm uu}}

\newcommand{\tz}{ \tilde{z}}

\newcommand{\tPhi}{\tilde{\Phi}}

\def\real{{\mathbb R}}

\def\tLambda{ {\tilde \Lambda}}

\def\hz{\hat{z}}
\def \t{\tilde{t}}

\begin{document}
\title[]
{Energy Transport in Random Perturbations of Mechanical Systems }
\author{Anna Maria Cherubini${^\dag}$}
\address{Dipartimento di Matematica e Fisica ``Ennio De Giorgi'', Universit\`a del Salento,
I-73100 Lecce, Italy}
\email{ anna.cherubini@unisalento.it}
\author{Marian\ Gidea$^\ddag$}
\address{Yeshiva University, Department of Mathematical Sciences, New York, NY 10016, USA }
\email{Marian.Gidea@yu.edu}
\thanks{$^\dag$ Research of A.M.C. was partially supported by GNFM/INDAM  }
\thanks{$^\ddag$ Research of M.G. was partially supported by NSF grant  DMS-2154725.}

\date{}

\begin{abstract}
We describe a mechanism for transport of energy in a mechanical system consisting of a pendulum and a rotator subject to a random perturbation.
The perturbation that we consider is the product of a Hamiltonian vector field and a scalar, continuous, stationary Gaussian process with H\"older continuous realizations, scaled by a smallness parameter. We show that for almost every realization of the stochastic process, there is a distinguished set of  times for which there exists a random normally hyperbolic invariant manifold  with associated stable and unstable manifolds that intersect transversally, for all sufficiently small values of the smallness parameter. We derive the existence of orbits along which the energy changes over time by an amount proportional to the smallness parameter. This result is related to the Arnold diffusion problem for Hamiltonian systems, which we treat here in the random setting.
\end{abstract}
\maketitle
\section{Introduction}

%

The main idea of the present work is using randomness to overcome geometric obstacles in
dynamical systems coming from classical mechanics, and, in particular, to generate energy transfer.
We  consider a  $2$-degrees of freedom,  uncoupled pendulum-rotator system, which is described  by an
integrable Hamiltonian.   The energies of the rotator and  of  the pendulum are first integrals of the system.
Hence, there are no trajectories that cross the level sets of the energies, so these are geometric obstacles for the dynamics.
We add a small, random perturbation to the system; we assume that this perturbation is of a special type.
More precisely, the perturbation is given by a Hamiltonian vector field multiplied by a scalar, continuous, stationary Gaussian process with H\"older continuous paths. Additionally, the Hamiltonian vector field is assumed to vanish at the hyperbolic fixed point of the pendulum.
The energies of the rotator and of  the pendulum are no longer conserved. We show that, in particular, the energy of the rotator can change by an amount proportional to the size of the perturbation.

Our treatment of the underlying random dynamics is path-wise, in the sense that we derive results for fixed realizations of the stochastic process. Each such realization is given by an unbounded,  continuous curve. Nevertheless, we cannot reduce the problem to the case of a non-autonomous perturbation by regarding each realization as a time-dependent parameter, since we need to use the ergodicity of the process and make use of the Birkhoff ergodic theorem.

Considering the effect of random perturbations of mechanical system is very natural in applications.
There are inherently many sources of noise that affect mechanical systems, or, more generally engineering systems.
Some concrete examples can be found in, e.g.,  \cite{skorokhod2007random,belbruno2008random,caraballo2017applied}.
One particular application that we plan to study in the future concerns piezoelectric energy harvesting devices, where one wants to exploit external vibrations
to generate electrical output; see \cite{erturk2009piezomagnetoelastic,granados2017invariant,akingbade2023arnold}. Many of the existing models assume that the external vibrations are periodic, but it would be more realistic to consider noisy perturbations. Such systems also include dissipation effects, yielding random attractors \cite{wang2012sufficient} and stochastic resonance \cite{cherubini2017random}.
For applications, the path-wise approach is suitable when we want to analyze the output of a single experiment at a time, rather than study the statistics of multiple experiments.

Our approach is based on  geometric methods. The unperturbed system possesses a normally hyperbolic invariant manifold (NHIM) whose stable and unstable manifolds coincide. To understand the effect of the time-dependent perturbation, we work in the extended space, where time is viewed as an additional coordinate.
We show that for a distinguished set of times, there  is a \emph{random normally hyperbolic invariant manifold} (RNHIM) and corresponding stable and unstable manifolds that survive the perturbation.
The main difference  from the standard normal hyperbolicity theory is  that our normally hyperbolic invariant manifold is time-dependent, and, moreover
we cannot guarantee its existence for all times, but only for a certain set of times.
These manifolds are reminiscent of Pesin sets in non-uniform hyperbolicity theory \cite{Pesin2007nonuniform}.
The reason for why the RNHIM's may break up is that the underlying stochastic process is unbounded, and large spikes in the noise can destroy the relations among the hyperbolic rates that are needed for normal hyperbolicity.

Results on the persistence of stable and unstable manifolds of hyperbolic fixed points under random perturbations have been obtained in  \cite{lu2010chaos,lu2011chaotic,yagasaki2018melnikov}.
To obtain a similar result for normally hyperbolic invariant manifolds, we apply some general results from  \cite{li2013normally,li2014invariant}; see also \cite{li2015geometric}. One difficulty in applying these results to our case is that they assume that the perturbed flow is close to the unperturbed flow at all times.
This is not the case when the noise driving the perturbation is unbounded. To deal with unbounded noise, we modify the system by multiplying the Hamiltonian vector field that appears in the perturbation by a \emph{random bump function} defined in the extended space.  This bump function depends on the noise parameter, and  has the effect of cutting off the  noise when it spikes too much. This makes the modified flow stay close to the unperturbed flow, so we can apply the general theory.  For the aforementioned  distinguished set of times (where time is viewed as an additional coordinate), we show that the modified flow coincides with the original flow. This is how we  obtain the persistence of the NHIM and of its stable and unstable manifolds  for the  distinguished set of times.

We also show that, for the distinguished set of times, the stable and unstable manifolds  intersect transversally provided that certain non-degeneracy conditions are satisfied. The transverse intersections of the manifolds correspond to non-degenerate  zeroes of a certain \emph{Melnikov process}.
A key ingredient to show the existence of zeroes for the Melnikov process is  \emph{Rice's formula}, concerning the number of times the process crosses a predetermined level. Melnikov theory for random perturbations has been developed in \cite{lu2010chaos,lu2010chaos}, and the idea of using Rice's formula to obtain intersections of the invariant manifolds comes from \cite{yagasaki2018melnikov}.

Once the existence of transverse intersections of the stable and unstable manifolds is established, the dynamics along the corresponding  homoclinic orbits can be described via a \emph{random scattering map}. This is an analogue of the (deterministic) scattering map developed in \cite{DelshamsLS06c}, and its  version for time-dependent system developed in \cite{BlazevskiL11}. This is a map defined on the  RNHIM, and gives the future asymptotic of a homoclinic orbit as a function of its the past asymptotic. We show that the random scattering map changes the energy of the rotator by an amount proportional to the size of the perturbation,  provided that some non-degeneracy conditions are satisfied.
These non-degeneracy conditions rely again on Rice's formula.

The change in energy owed to the small perturbation is reminiscent of the \emph{Arnold diffusion problem} for Hamiltonian systems \cite{Arnold64}. Arnold conjectured that integrable Hamiltonian systems of general type, of more than two degrees of freedom,  subjected to small, Hamiltonian perturbations of generic type have  trajectories along which the  energy  changes by an amount that is independent of the perturbation parameter. A survey on some recent results can be found in \cite{gidea2020general}.
Much of the existing work considers deterministic perturbations.
Diffusion in randomly perturbed integrable Hamiltonian systems has been studied in \cite{bazzani1994diffusion,bazzani1998diffusion},
where they derive the Fokker-Planck equation for the distribution function of the action angle-variables.
A model for diffusion for  random compositions of cylinder maps was considered in \cite{castejon2017random}.
Another paper of related interest is \cite{DAmbrosio23}.

The upshot of our work is that, we can extend the geometric methods for Arnold diffusion developed in \cite{DelshamsLS00,DelshamsLS00,DelshamsLS06b,gelfreich2017arnold,gidea2020general} to the case in which the perturbations are random (rather than  deterministic).
The main difficulty in applying the geometric method is that the perturbation is driven by unbounded  noise, so the spikes in the noise may destroy the geometric structures.
Our current results  yield only a small change in energy, of the order of the perturbation (that is, we obtain micro-diffusion -- a term coined in \cite{bounemoura2016note}), rather than of order one.
The identification of orbits with  diffusion of order one  is the object of future investigation.


\section{Set-up}\label{section:setup}
\subsection{Unperturbed system}

The unperturbed system  is a rotator-pendulum system described by an autonomous Hamiltonian $H_0$ of the form:
\begin{equation}\label{eqn:rotator_pendulum}
\begin{split}
H_0(I,\phi,p,q)&=h_0(I)+h_1(p,q)\\&=h_0(I)+ \left (\frac{1}{2}p ^2+V (q )\right),
\end{split}
\end{equation}
with $z=(I,\phi,p,q)$ in $M:=\R\times \T\times \R\times \T$.

The phase space $M$ is endowed with the symplectic form
\[ dI\wedge d\phi+dp\wedge dq.\]

We denote by \[\nu(I):=\frac{\partial h_0}{\partial I}(I)\] the frequency of the rotator.

We assume the following:
\begin{itemize}
\item[(P-i)]
The potential $V$ is periodic of period $1$ in~$q$;
\item[(P-ii)] The potential $V$ has a non-degenerate local  maximum,  which, without loss of generality,
we set at $0$; that is,   $V'(0)=0$ and  $V''(0)<0$. We additionally assume that $q=0$ is non-degenerate in the sense of Morse, i.e.,  $0$ is the only critical point in the level set $\{V(q) = V (0)\}$.
\end{itemize}

Condition (P-ii) implies that the
pendulum   has a homoclinic orbit to $(0, 0)$, the hyperbolic fixed point of the pendulum. We consider that the homoclinic orbit is parametrized by $(p_0(t),q_0(t))$ for $t\in \mathbb{R}$, where $(p_0(t),q_0(t))\to (0,0)$ as $t\to\pm\infty$.

The Hamilton equation  associated to \eqref{eqn:rotator_pendulum} is
\begin{equation}\label{eqn:hamiltonian_unperturbed}
    \dot{z}=X^0(z)=J\nabla H_0(z),
    \end{equation}
where  $J$ is the symplectic matrix
\[J=\left(
    \begin{array}{cc}
      J_2 & 0  \\
      0 &  J_2  \\
      \end{array}
  \right)\textrm{ with } J_2=\left(
    \begin{array}{rr}
      0 & -1 \\
      1 &  0 \\
      \end{array}
  \right).\]

We denote by $\Phi^t_0$ the flow of \eqref{eqn:hamiltonian_unperturbed}.

Since for $H_0$ the pendulum and the rotator are decoupled, the action variable $I$ is preserved along the trajectories of \eqref{eqn:hamiltonian_unperturbed}.
Similarly, the energy $P=\left (\frac{1}{2}p ^2+V (q )\right)$ of the pendulum is a conserved quantity.

In the sequel, we will show that if we add a small, random perturbation to the pendulum-rotator system, there are trajectories of the perturbed system along which $I$ changes over time. If  $\nu(I)=\frac{\partial h_0}{\partial I}(I)\neq 0 $ for all $I$ within some range, the fact that $I$ changes along a trajectory implies that the energy of the rotator $h_0(I)$ also changes along that trajectory.

\subsection{Perturbed system}
To the system \eqref{eqn:hamiltonian_unperturbed} we add a random perturbation, so that the perturbed system is of the form
\begin{equation}\label{eqn:hamiltonian_perturbed}
\begin{split}\dot{z}=& X_\eps(z,\eta(t))=
X^0(z)+\eps X^1 (z,\eta(t))\\
=&J\nabla H_0(z)+\eps J\nabla H_1(z)\eta(t)
\end{split}
\end{equation}
for $\eps\in\R$, where $\eta(t)$ is a scalar, continuous, stationary Gaussian process satisfying the properties \eqref{eqn:noise_1}, \eqref{eqn:noise_2}, \eqref{eqn:noise_3}  below.

In the above,  we assume that $H_1(z )$ is a   Hamiltonian function, uniformly $\C^2$ in $z$, satisfying the following condition
\begin{equation}\begin{split} H_1(I,\phi,0,0 )&=0,\, DH_1(I,\phi,0,0 ) =0.\end{split}\label{eqn:H_1}\tag{H1}\end{equation}

The perturbation is chosen so does not affect the
inner dynamics, given by the restriction to the phase space of
the rotator. The dynamics of the rotator is integrable, hence $I$
is preserved by the inner dynamics.

The level sets of $I$ (which are invariant circles) constitute geometric obstacles for the inner dynamics. We will show that we can use the outer dynamics, along the homoclinic orbits of the pendulum, to overcome these geometric obstacles.

The system \eqref{eqn:hamiltonian_perturbed} is non-autonomous.
We denote by $\Phi^{t_0,t}_\eps(\omega)$ the corresponding flow, which depends on the initial time $t_0$ and on the realization $\omega$ of the stochastic process $\eta(t)$. For every fixed  realization $\omega$  of $\eta(t)$ we have a sample path given by  $\eta(t)(\omega)=\omega(t)$. See Section \ref{sec:RODE}.

\begin{rem}
It may be possible to remove condition \eqref{eqn:H_1}.
In fact, we will not use this condition for two of the main results (Theorem \ref{thm:random_NHIM} and Theorem \ref{prop:transverse}.
However, without \eqref{eqn:H_1}, the inner dynamics will be affected by the random perturbation, and the resulting inner dynamics may overcome on its own the geometric obstacles.
\end{rem}

\subsection{Noise}
The time-dependent function $\eta (t)$ is a scalar stationary Gaussian process with mean $0$, i.e.,
\begin{equation}\label{eqn:noise_1}\tag{R1} E[\eta(t)] = 0.\end{equation}
Stationarity means that for any $n$, $t_1,\ldots, t_n$,  and $h>0$, the random vectors
\[ (\eta(t_1),\ldots, \eta(t_n)) \textrm { and }  \textrm (\eta(t_1+h),\ldots, \eta(t_n+h))\]
have the same (Gaussian) distribution.

The  autocorrelation
function $r(h)$ is
\begin{equation}\label{eqn:autocorrelation}
r(h):=E[\eta(t)\eta(t +h )].
\end{equation}
By stationarity, the right-hand side of \eqref{eqn:autocorrelation} does
not depend on $t$.

We assume  that the autocorrelation function satisfies the following conditions
\begin{equation}\label{eqn:noise_2}\tag{R2} r(h )\textrm { is   continuous and absolutely integrable on }\mathbb{R},\end{equation}
and
\begin{equation}\label{eqn:noise_3}\tag{R3}\begin{split} r(h )=&1-C|h|^{a}+o(|h|^a)  \textrm{ as $h\to 0$,}
\\
\textrm{for some $C$ with } &\textrm{$0< C < \infty$, and some $a$ with $ 1<a \leq 2 $}.\end{split}\end{equation}

We use the notation $f(x)=o(g(x))$ to signify $\lim_{x\to 0}\frac{f(x)}{g(x)}=0$.

For a Gaussian process, if $r(h)$ satisfies \eqref{eqn:noise_3} for $0<a\le 2$, then  the sample paths  $\omega(t)$ are continuous, and if $1<a\le 2$,
then  the sample paths  $\omega(t)$ are $\alpha$-H\"older continuous for any $0<\alpha<\frac{a-1}{2}$. See \cite{belyaev1961continuity,lindgren2012stationary}.
That is,
there exists $C_H>0$, independent of $\omega$, such that
\begin{equation}\label{eqn:Holder}
 |\omega(t_1)-\omega(t_2)|<C_H|t_1-t_2|^\alpha  \textrm{ for all }t_1,t_2\in\mathbb{R}.
\end{equation}

Intuitively, the above conditions say that the lesser the loss of memory of the process is, the more regular the sample paths of the process are.

To summarize, by assumption  \eqref{eqn:noise_3},  the sample paths of  $\eta(t)$ are  continuous and $\alpha$-H\"older continuous with probability $1$.

Condition \eqref{eqn:noise_3} yields $r(0)=1$, and therefore
\begin{equation}\label{eqn:variance}
 E[\eta(t)^2]=1,
\end{equation}
which means that the  Gaussian process has variance equal to $1$.

Also,  \eqref{eqn:noise_3} implies, via the Maruyama Theorem  \cite{maruyama1949harmonic}, that $\eta(t)$ is ergodic
\begin{equation}\label{eqn:ergodic}
\lim_{T\to\pm\infty}\frac{1}{T}  \int_{0}^{T} \phi(\eta(t))dt = E[\phi(\eta(t))],\, \forall\phi:\mathbb{R}\to\mathbb{R}\textrm{ measurable function}.
\end{equation}

\begin{rem}
Gaussian stationary processes that are not continuous are necessarily very irregular.
More precisely, one of the following alternatives holds: either with probability one the sample paths $\omega(t)$ are continuous, or with probability one they are unbounded on every finite interval \cite{belyaev1961continuity}.
Hence, considering  Gaussian processes with continuous sample paths as in \eqref{eqn:noise_3} is a reasonable assumption.
\end{rem}

\section{Main results}\label{sec:main_results}

In Section \ref{section:NHIM_ unperturbed rotator_pendulum} we show that the unperturbed rotator-pendulum system possesses a normally hyperbolic invariant manifold (NHIM).
The first main result is the persistence of the NHIM, as  random normally hyperbolic invariant manifold (RNHIM), and of its stable and unstable manifolds, for a distinguished  set of times. The definition of an RNHIM and of its stable and unstable manifolds is given in Section \ref{sec:RNHIM}.

\begin{thm}[Persistence of RNHIM]\label{thm:random_NHIM}
Assume that the system \eqref{eqn:hamiltonian_perturbed} satisfies (P-i), (P-ii), (R-i), (R-ii), (R-iii) (but not necessarily (H1)).

Then, for any $\delta>0$, there exist a positive random variable $T_\delta(\omega)$, a closed set  $Q_{A_\delta,T_\delta}(\omega)\subseteq [-T_\delta, T_\delta]$,
and  $\eps_0>0$ such that, for every $t_0\in Q_{A_\delta,T_\delta}(\omega)$,   every $\eps$ with $0<\eps<\eps_0$, and a.e. $\omega \in\Omega$
there exist the following objects:
\begin{itemize}
\item[(i)] A normally hyperbolic manifold $\Lambda_\eps(\theta^{t_0}(\omega))$ 
such that
\begin{equation}\Phi^{t_0,t_1}_\eps  (\Lambda_\eps(\theta^{t_0}(\omega)))=\Lambda_\eps(\theta^{t_0+t_1}(\omega)),\end{equation}
 provided that $t_0+t_1 \in  Q_{A_\delta,T_\delta}(\omega)$.
\item[(ii)] Stable and unstable manifolds $W^{\st}(\Lambda_\eps(\theta^{t_0}(\omega)))$ and $W^{\un}(\Lambda_\eps(\theta^{t_0}(\omega)))$
such that
 \begin{equation}\label{}
   \begin{split}
   \Phi^{t_0,t_1}_\eps(\theta^{t_0}(\omega))(W^{\st}  (\Lambda_\eps(\theta^{t_0}(\omega))))=&W^{\st}(\Lambda_\eps(\theta^{t_0+t_1}(\omega)),\\
   \Phi^{t_0,t_1}_\eps(\theta^{t_0}(\omega))(W^{\un}  (\Lambda_\eps(\theta^{t_0}(\omega))))=&W^{\un}(\Lambda_\eps(\theta^{t_0+t_1}(\omega)),
   \end{split}
 \end{equation}
   provided that $t_0+t_1 \in  Q_{A_\delta,T_\delta}(\omega)$.
\end{itemize}
See Fig. \ref{fig:RNHIM}.

Moreover, the distinguished set of times
  $Q_{A_\delta,T_\delta}(\omega)$ satisfies the conditions \eqref{eqn:Q_A_measure} and \eqref{eqn:Q_A_monotone} given in Section \ref{sec:sublinearity}.
\end{thm}

\begin{figure}
\includegraphics[width=0.6\textwidth]{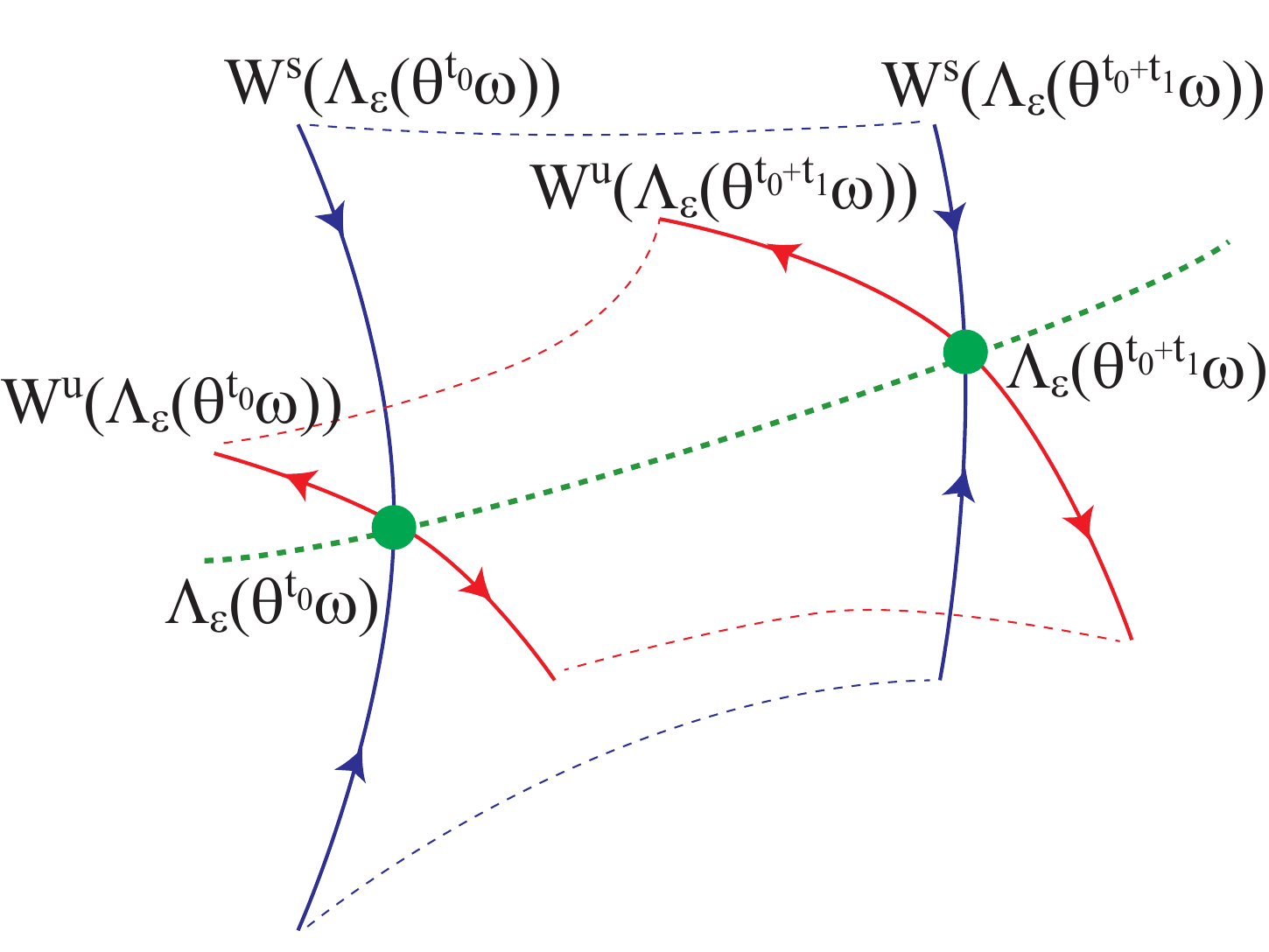}
\caption{Random NHIM and its stable and unstable manifolds.}
\label{fig:RNHIM}
\end{figure}

In Section \ref{sec:existence_transverse} we define a  Melnikov process $M^P$ \eqref{eqn:melnikov_process} to measure the splitting of the perturbed stable and unstable manifolds as it changes over time.
The second main results says that, if the Melnikov process satisfies some non-degeneracy conditions,  then the stable and unstable manifolds, corresponding to the distinguished set of times provided by Theorem \eqref{thm:random_NHIM}, intersect transversally.

\begin{thm}\label{prop:transverse}
Assume that the system \eqref{eqn:hamiltonian_perturbed} satisfies (P-i), (P-ii), (R-i), (R-ii), (R-iii) (but not necessarily (H1)).
Also assume the condition \eqref{eqn:spectral_moments} on the spectral moments of $M^P$, given in Section \ref{sec:existence_transverse}.

Then, given a set $Q_{A_\delta,T_\delta}(\omega)$ and $\eps_0>0$ as in Theorem \ref{thm:random_NHIM},  there exist $t_0\in Q_{A_\delta,T_\delta}(\omega)$ and $0<\eps_1<\eps_0$ such that for   every  $0<\eps<\eps_1$, the stable and unstable manifolds $W^\st(\Lambda_\eps(\theta^{t_0}\omega))$ and $W^\un(\Lambda_\eps(\theta^{t_0}\omega))$ have a transverse intersection at a point $\tz_\eps$.
\end{thm}


In Section \ref{sec:random_scattering} we define the random scattering map associated to a homoclinic intersection of the stable and unstable manifolds obtained in Theorem \ref{prop:transverse}. This map relates the future asymptotic of a homoclinic orbit as a function of its the past asymptotic.

In Section \ref{sec:change_in_action} we define a Melnikov process $M^I$ \eqref{eqn:tau_I_integral} to measure the splitting of the action $I$-level sets by the scattering map.
The third main result says that, if the Melnikov process satisfies some non-degeneracy conditions, then the scattering map grows the action $I$ (and hence the energy of the rotator) by $O(\eps)$. Consequently, there are trajectories of the perturbed system along which the action $I$ (energy)  grows by $O(\eps)$.

\begin{thm}\label{prop:change_in_I}
Assume that the system \eqref{eqn:hamiltonian_perturbed} satisfies (P-i), (P-ii), (R-i), (R-ii), (R-iii) and  (H1).
Also assume the condition \eqref{eqn:spectral_moments_I} on the spectral moments of $M^I$, given in Section \ref{sec:change_in_action}.

Then, given
$\eps_0>0$ as in Theorem \ref{thm:random_NHIM}, for every $v>0$   there exist
%
%
%
trajectories $\tz_\eps(t)$, and  times $T_\eps>0$,  such that
\[ I(\tz_\eps(T_\eps))-I(\tz_\eps(0))=\eps v +O(\eps^{1+\rho}).\]
\end{thm}

This result says that the perturbed system has trajectories that exhibit micro-diffusion in the action variable.
The obtained change in action is obtained along a single homoclinic orbit.
We stress that, while the change is of order
$O(\eps)$, the constant $v$ can be chosen arbitrarily large. This is different from the case of time-periodic (or quasi-periodic) Hamiltonian perturbations.

It seems possible to obtain true Arnold diffusion, i.e.,  existence of trajectories along which the action variable changes by
$O(1)$, by concatenating   trajectories segments that change $I$ by $O(\eps)$, and then applying a shadowing lemma similar to those in \cite{gelfreich2017arnold,gidea2020general}. Some of the challenges is to show that there are $O(1/\eps)$ such trajectories  segments, and to obtain a version of the aforementioned shadowing lemma in the random setting.

\section{Background}

%
%

\subsection{Random dynamical system}\label{sec:RDS}
We consider a probability space  $(\Omega, \mathscr{F},\mathbb{P})$, where $\Omega$ is the sample space of
outcomes, $\mathscr{F}$ is the $\sigma$-algebra  of events,  and $\mathbb{P}$ is a probability measure
 that assigns probabilities to the events in $\mathscr{F}$.

A stochastic process $\eta$ is a mapping $t\in\real\mapsto \eta(t):\Omega \to R$, where
each $\eta(t)$ is a random variable, i.e., a measurable function from $\Omega$ to $\real$.
For a fixed realization $\omega\in\Omega$, the function $\omega(t):=\eta(t)(\omega)$ is referred to as a  sample path.
(We note that, in this paper, under \eqref{eqn:noise_3}, we assume that a.e. sample path is H\"older continuous.)

On $\Omega$ we consider the $\mathbb{P}$-preserving measurable flow  $\theta^t:\Omega\to\Omega$, given by
\begin{equation}
\theta^t \omega (s)=\omega(t+s).
\end{equation}
It satisfies  the following conditions for all   $\omega\in\Omega$:
\begin{itemize}
\item[(i)] $\theta^0  \omega =\omega$,
\item[(ii)]  $\theta^{t_1+t_2} \omega =\theta^{t_1}(\theta^{t_2} \omega)$, for $t_1,t_2\in\mathbb{R}$,
\end{itemize}for each $\omega\in\Omega$.
The dynamical system $(\Omega,\theta^t)$ is referred to as a \emph{metric dynamical system}.

Let $(M,\mathscr{B})$ be a measurable space.
A measurable mapping  \[\Phi:\mathbb{R}\times \Omega \times M\to M\] is a \emph{random dynamical system (RDS) over $\theta$}, if  it satisfies the following cocycle conditions for all   $\omega\in\Omega$:
\begin{itemize}
\item[(i)] $\Phi(0,\omega)(z)=(z)$, for all $z\in M$,
\item[(ii)] $\Phi(t_1+t_2,\omega)(z)=\Phi(t_1,\theta^{t_2}\omega)(\Phi(t_2,\omega)(z))$, for all $z\in M$, $t_1,t_2\in \R$.
\end{itemize}

We will often write $\Phi(t,\omega)=\Phi^t(\omega)$.

\subsection{Random differential equations}\label{sec:RODE}
Let $M$ be a smooth manifold and $X:M\times\Omega \to TM$, $(z,\omega)\in M\times\Omega \mapsto X(z,\omega)\in T_zM$,  be a random vector field that is  $\C^r$, $r\ge 1$, in the $z$-component, and measurable in the $\omega$-component.

A solution (in the sense of Carath\'eodory) of a  system
 \begin{equation}\label{eqn:equation_Caratheodory}
\left\{
 \begin{array}{ll}
    \dot{z} &= X(z,\theta^t\omega) \\
    z(t_0)  &= z_0
  \end{array}
\right.
\end{equation}
is a function  $z(t;z_0,t_0,\omega):M\to M$ satisfying
\begin{equation}\label{eqn:Caratheodory}
z(t;z_0,t_0,\omega)=z_0+\int_{t_0}^{t} X\left(z(s;z_0,t_0,\omega),\theta^s\omega\right)ds,
\end{equation} where we  fix a realization $\omega$ of the  process $\eta(t)$.
In the above notation $\omega$ refers to the choice of realization at time $t=0$.

If for all $\omega\in\Omega$, $X_\eps\in  \C^{r}$, $r\geq 1$, then  \eqref{eqn:hamiltonian_perturbed}  yields a unique solution $t\mapsto z(t;z_0,t_0,\omega)$ which is $\C^r$ in $z_0$.
See \cite{arnold1998random,duan2015introduction}.

If we denote $z(t)=z(t;z_0,t_0,\omega)$ and $\bar{z}(t)=z(t_0+t)$, and make a change of variable $s\mapsto t_0+s'$ in \eqref{eqn:Caratheodory},
we have
\begin{equation}\label{eqn:change_of_variable}
\begin{split}
z(t)=&z_0+\int_{0}^{t-t_0} X\left(z(t_0+s'),\theta^{t_0+s'}\omega\right)ds'\\
=&z_0+\int_{0}^{t-t_0} X\left(\bar z(s'),\theta^{s'}(\theta^{t_0}\omega)\right)ds'.
\end{split}
\end{equation}
Noting that $\bar{z}(0)=z(t_0)=z_0$, we see that the right hand side of \eqref{eqn:change_of_variable} is the solution of
 \begin{equation}\label{eqn:equation_Caratheodory}
\left\{
 \begin{array}{ll}
    \dot{\bar z} &= X(\bar z,\theta^t(\theta^{t_0}\omega)) \\
    \bar z(0)  &= z_0
  \end{array}
\right.
\end{equation}
evaluated at time $t-t_0$, therefore, it coincides with $z(t-t_0; z_0,0,\theta^{t_0}\omega)$.

We obtained  the  following invariance relation:
\begin{equation}\label{eqn:elapsed_time}
 z(t;z_0, t_0,\omega)=z(t-t_0;z_0,0,\theta^{t_0}\omega).
\end{equation}
That is, the solution only depends on the elapsed time $t-t_0$ and on the  random parameter $\theta^{t_0}\omega$ at time $t_0$.

Therefore,
\begin{equation}\label{eqn:RDE_RDS}
  \Phi(t,\omega)(z_0)=z(t;z_0,0,\omega)
\end{equation}
determines a \emph{random dynamical system (RDS) over $\theta^t$}.

From \eqref{eqn:elapsed_time} and \eqref{eqn:RDE_RDS} we have that, for any $t_0$,
\begin{equation}\label{eqn:RDE_RDS_2}
  \Phi(t-t_0,\theta^{t_0}\omega)(z_0)=z(t-t_0;z_0,0,\theta^{t_0}\omega)=z(t;z_0,t_0,\omega).
\end{equation}

In the case when  the sample paths $\omega(t)$ are continuous,  for each realization $\omega$ the equation \eqref{eqn:equation_Caratheodory} is a classical non-autonomous differential equation,
and its solutions  \eqref{eqn:Caratheodory} are in  classical sense.

\subsection{Normally hyperbolic invariant manifolds for random dynamical systems}\label{sec:RNHIM}
In the sequel we follow \cite{li2013normally,li2014invariant}.

Let $\Phi(t,\omega)$ be a random dynamical system.

A random set is a mapping \[\omega\in\Omega\mapsto \mathcal{M}(\omega)\subseteq M\]
assigning to each path $\omega\in\Omega$ a closed subset $\mathcal{M}(\omega)\subseteq M$,
such that
\[\omega \to \inf_{y\in \mathcal{M}(\omega)} \| y-x\| \textrm { is measurable for each } x \in M.\]

A random invariant manifold  is a random manifold $\mathcal{M}(\omega)\subseteq M $ such that
\[\Phi(t,\omega)(\mathcal{M}(\omega))=\mathcal{M}(\theta^t\omega), \textrm{ for all } t\in\mathbb{R}, \omega \in\Omega.\]

A random variable $C(t)$ is said to be {\em  tempered} if
\begin{equation}\label{eqn:tempered}
\lim_{t\to\pm\infty}\frac{\log C(\theta^t \omega )}{t}=0 \textrm{ for a.e. } \omega \in\Omega.
\end{equation}

\begin{defn}\label{defn:RNHIM} A random invariant manifold $\Lambda(\omega)$ is  normally hyperbolic
if for a.e. $\omega \in\Omega$ and all $x\in\Lambda(\omega)$ there exists an invariant splitting of $T_xM$, which is $\C^0$ in $x$ and measurable in $\omega$,
\[T_xM=T_x\Lambda(\omega)\oplus E^{\un}_x(\omega)\oplus E^{\st}_x(\omega),\]
whose bundles are invariant in the sense
\begin{equation*}\begin{split}
D_x\Phi(t,\omega)(T_x\Lambda(\omega))=& T_{\Phi(t,\omega)(x)}\Lambda(\theta^t\omega) ,\\
D_x\Phi(t,\omega)(E^{\un}_x (\omega))=&E^{\un}_{\Phi(t,\omega)(x)}(\theta^t\omega),\\
D_x\Phi(t,\omega)(E^{\st}_x (\omega))=&E^{\st}_{\Phi(t,\omega)(x)}(\theta^t\omega),
\end{split}\end{equation*}
and  there exist a tempered random variable $C(x,\omega)>0$ and $(\theta, \Phi)$-invariant random variables (rates)
\[ 0<\alpha (x,\omega)<\beta(x,\omega)\]
such that for all $x\in\Lambda(\omega)$ we have
\begin{equation}
\label{eqn:NHIM_rates}
\begin{split}
v\in E^{\st}_x(\omega)\Rightarrow\|D_x\Phi(t,\omega)v\|<&C(x,\omega)e^{-\beta(x,\omega)t}\|v\|, \textrm { for } t\geq 0,\\
v\in E^{\un}_x(\omega)\Rightarrow\|D_x\Phi(t,\omega)v\|<&C(x,\omega)e^{ \beta(x,\omega)t}\|v\|, \textrm { for } t\leq 0,\\
v\in  T_x\Lambda(\omega)\Rightarrow\|D_x\Phi(t,\omega)v\|<&C(x,\omega)e^{ \alpha(x,\omega)|t|}\|v\|, \textrm { for all } t.
\end{split}
\end{equation}
\end{defn}

When the objects in Definition \ref{defn:RNHIM} do not depend on the random parameter $\omega$, we obtain the classical definition of a NHIM as in \cite{HirschPS77}; a brief summary of the normal hyperbolicity theory can be found in \cite{DelshamsLS00}.

We will use the following results from \cite{li2013normally,li2014invariant} on the persistence of the NHIM under  random perturbations in Section \ref{sec:persistence_NHIM}.

\begin{thm}[Persistence of NHIM]\label{thm:Bates1}
Assume that $\Phi_0(t)$  is a (deterministic) $\C^r$ flow on $M$, $r \geq 1$. Assume that $\Phi_0(t)$ has a compact, connected
$\C^r$ normally hyperbolic invariant manifold $\Lambda\subseteq M$. Let the positive exponents
related to the normal hyperbolicity of $\Lambda$ be $0<\alpha < \beta$,  which  are
constant and deterministic.

Then there exists $\bar{\eps}>0$ such that for any random
flow $\Phi(t, \omega)$ on $M$ which is $\C^1$ in $x$, if
\begin{equation}\label{eqn: flow_error}
\|\Phi(t, \omega)- \Phi_0(t)\|_{\C^1} < \bar{\eps}, \textrm { for } t \in [0, 1],\, \omega \in\Omega\end{equation}
then
\begin{itemize}
\item [(i)] {\bf Persistence:} $\Phi(t, \omega)$ has a normally hyperbolic random invariant manifold
$\Lambda(\omega)$ which is $\C^1$ in $x$;
\item [(ii)] {\bf Smoothness:} If $\ell<\min\{\beta/\alpha,r\}$,  then $\Lambda(\omega)$ is   $\C^\ell$ in $x$, and is
diffeomorphic to $\Lambda$ for a.e. $\omega\in\Omega$;
\item [(iii)] {\bf Existence of stable and unstable manifolds:} $\Phi(t, \omega)$ has stable manifolds $W^{\st}(\Lambda(\omega))$ and $W^{\un}(\Lambda(\omega))$ that are $\C^{\ell-1}$ and depend measurably on $\omega$;
\item[(iv)]{\bf Existence of stable and unstable foliations:} The stable manifold  $W^{\st}(\Lambda(\omega))$ is foliated by
an equivariant family of $\C^r$ stable leaves
\[ W^{\st}(\Lambda(\omega))=\bigcup_{x\in \Lambda(\omega)} W^{\sst}(x,\omega),\]
and the unstable manifold  $W^{\un}(\Lambda(\omega))$ is foliated by
an equivariant family of $\C^r$ unstable leaves
\[ W^{\un}(\Lambda(\omega))=\bigcup_{x\in \Lambda(\omega)} W^{\uun}(x,\omega),\]
both depending measurably on $\omega$.
\end{itemize}
\end{thm}

Condition \eqref{eqn: flow_error} implies that the perturbed flow and the unperturbed flow stay $\bar\eps$-close for all time, since we can
re-initialize the time and update the random variable at the end of the  time-interval $[0,1]$.
Also note that we can replace the time domain $[0,1]$ with any closed interval.

The perturbed NHIM, and its stable and unstable manifolds, can be described as graphs over the unperturbed ones, respectively.
There exist a smooth parametrization $k_{\omega}:\Lambda_0\to\Lambda(\omega)$, depending on $\omega$ in a measurable fashion, such that $\Lambda(\omega)=k_{\omega}(\Lambda_0)$.
Given a system of coordinates on the unperturbed manifold $\Lambda_0$, we can transport it through $k_\omega$ to obtain a system of coordinates on $\Lambda(\omega)$. We will use this fact in Section \ref{sec:transverse}.

The stable and unstable manifolds  $W^{\st}(\Lambda(\omega))$,  $W^{\un}(\Lambda(\omega))$, given in Theorem \ref{thm:Bates1}, have the following asymptotic properties:
\begin{equation}\label{eqn:stable_unstable_equivalent}
\begin{split}
  x \in W^{\st}(\Lambda(\omega))  \Rightarrow & \exists x^+ \in\Lambda(\omega) \textrm{ s.t. } d(\Phi(t,\omega)(x),\Phi(t,\omega)(x^+))\to 0\textrm { as } t\to+\infty,\\
  x \in W^{\un}(\Lambda(\omega))  \Rightarrow & \exists x^-\in \Lambda(\omega) \textrm{ s.t. } d(\Phi(t,\omega)(x),\Phi(t,\omega)(x^-))\to 0\textrm { as } t\to -\infty,
\end{split}
\end{equation}
where the point $x^\pm \in\Lambda(\omega)$ is uniquely defined by $x$.
Then, we respectively have
\begin{equation}\label{eqn:stable_unstable_fiber_equivalent}
\begin{split}
  x \in W^{\st}(\Lambda(\omega))  \Rightarrow & x \in W^{\sst}(x^+,\omega),\\
  x \in W^{\un}(\Lambda(\omega))  \Rightarrow & x \in W^{\uun}(x^-,\omega).
\end{split}
\end{equation}

\section{Preliminary results}

\subsection{Sub-linearity of the noise }\label{sec:sublinearity}
We assume that the stochastic process $\eta$ satisfies \eqref{eqn:noise_1}, \eqref{eqn:noise_2}, \eqref{eqn:noise_3}.
\begin{lem}\label{lem:eta}
For almost every realization $\omega\in\Omega$ of $\eta$ we have
\begin{equation}\label{eqn:eta_lim}\lim_{s\to+\infty}\frac{|\omega(s)|}{s}=0,\end{equation} that is, $\omega(s)=o(s)$.
Therefore, there exist  $A_\omega >0$ depending on $\omega$  and $B>0$ that can be chosen independent of $\omega$, such that
\begin{equation}\label{eqn:eta_ineq}|\omega(s) |\le  A_\omega+Bs\textrm{ for all } s\ge 0. \end{equation}
See Fig.~\ref{fig:Sets_1}
\end{lem}
\begin{proof}
Since $\eta(s)$  is a continuous stationary Gaussian process with $E[\eta(s)^2] = 1$, it follows  from    \cite[Theorem 1.4]{marcus1972upper}  that for every realization $\omega$ we have \[\limsup_{s\to +\infty}\frac{|\omega(s)|}{\sqrt{2\log(s)}}\leq 1.\]
This implies that \[\lim_{s\to+\infty}\frac{|\omega(s)|}{s}=0.\]
Since for almost every $\omega\in\Omega$, $\omega$ is continuous in $s$, then for any given $B>0$ there exists $T_{\omega}>0$ such that $\frac{|\omega(s)|}{s}<B$ for all $s\geq T_{\omega}$.
Let $A_\omega>\sup \{|\omega(s)|\,|\, s\in [0,T_{\omega}]\}$.  Then $|\omega(s)|\le A_\omega+Bs$ for all $s\geq 0$.
Note that $A_\omega$ depends on $\omega$, while $B$ does not; moreover, $B>0$ can be chosen arbitrarily small.
\end{proof}

\begin{figure}
\includegraphics[width=0.5\textwidth]{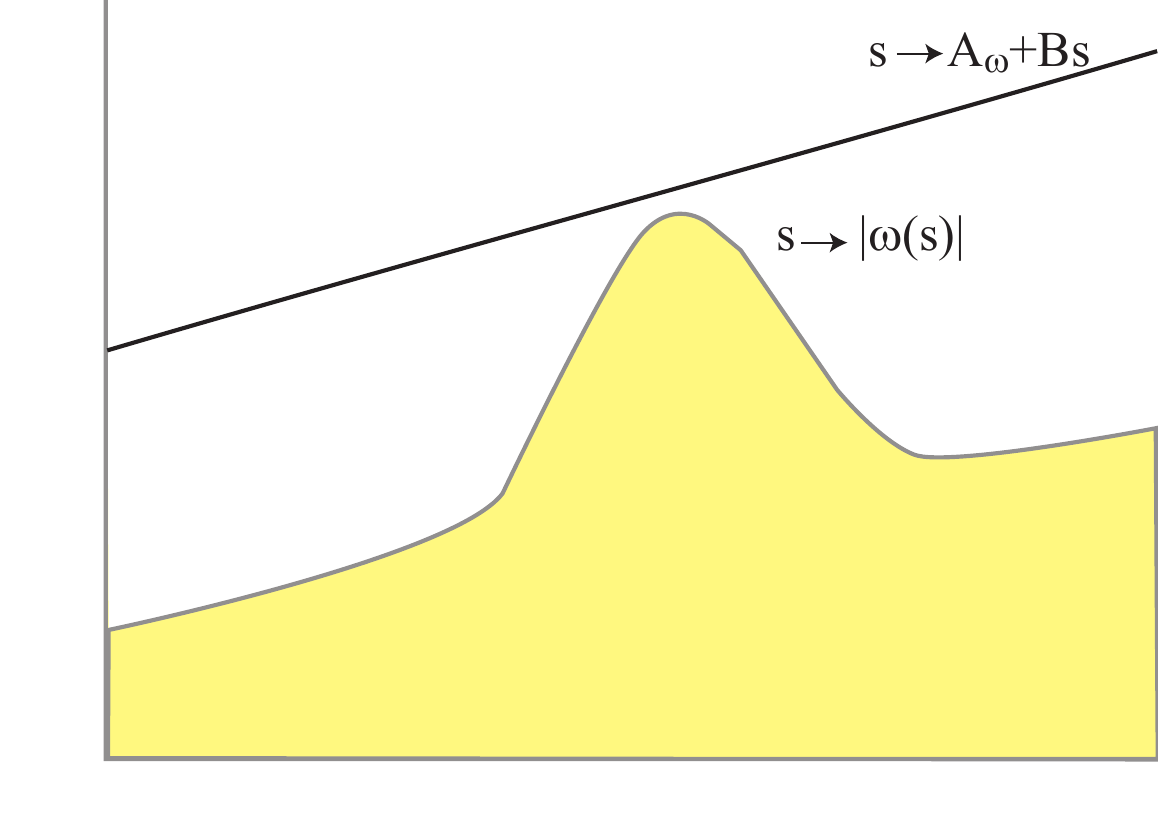}
\caption{Sub-linearity of the noise}
\label{fig:Sets_1}
\end{figure}

The meaning of Lemma \ref{lem:eta} is that the graph of $|\omega(s)|$ for $s\ge 0$ is below the line $s\mapsto  A_\omega+Bs$, where $B$ is a fixed slope independent of the path $\omega$, and $A_\omega$ is the vertical intercept  of the line and depends on the path. The slope $B$ can be chosen arbitrarily close to $0$, and is fixed once and for all.

We define some relevant sets and describe relations among them.

A consequence of Lemma \ref{lem:eta} is that for a.e. $\omega\in \Omega$, there exist $B>0$ and  $A=A_\omega>0$ such that
 \begin{equation}\label{eqn:eta_ineq_abs} |\omega(s) |\le A_\omega+B|s| \textrm{ for all } s\in\mathbb{R}.  \end{equation}
Without any loss of generality, by disregarding a measure zero set of paths,  we can assume that this property is true for all $\omega\in\Omega$.

For any $A>0$ fixed, define the set:
\begin{equation}\label{eqn:Omega_A}
\Omega_A=\{\omega\in\Omega\,|\,|\omega(s) |\le A+B|s|,\forall s\in\mathbb{R}\}.
\end{equation}

Since $s\mapsto |\omega(s)|$ is a continuous function, the set $\Omega_A$ is a measurable subset of $\Omega$, and
\[A_1\le A_2\implies \Omega_{A_1}\subseteq \Omega_{A_2}.\]
Since $\bigcup_{A>0}\Omega_{A}=\Omega$, from the continuity from below of the measure $\mathbb{P}$, 
we have that $\lim_{A\to\infty} \mathbb{P}(\Omega_A)=\mathbb{P}(\Omega)=1$, therefore
\begin{equation}\label{eqn:Omega_A_delta}
\forall \delta>0,\,\exists A_\delta>0 \textrm{ s.t. } \mathbb{P}(\Omega_A)>1-\delta.
\end{equation}

Note that if $\omega \in\Omega_A$ and $t\neq 0$, it does not follow that $\theta^t\omega\in\Omega_A$.
It  is also possible that $\omega \not\in \Omega_A$ and for some $t\neq 0$ we have $\theta^t\omega\in\Omega_A$.
That  is to say that  the sets $\Omega_A$ are not invariant under the shift $\theta^t$. See  Fig.~\ref{fig:noise_bounds}.

\begin{figure}
\includegraphics[width=0.95\textwidth]{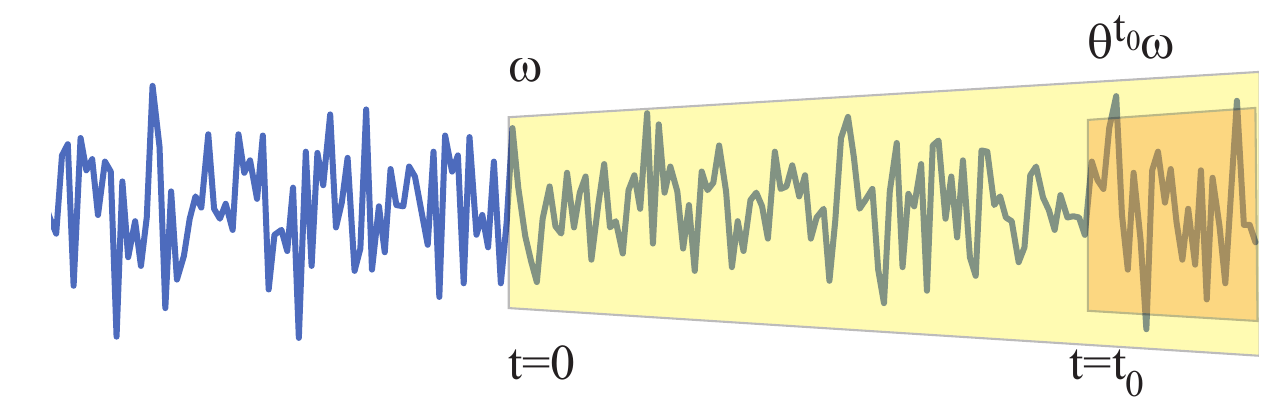}
\caption{The path  $\omega$  is inside $\Omega_A$ while the path $\theta^{t_0}\omega$ is not inside $\Omega_A$.}
\label{fig:noise_bounds}
\end{figure}

On the other hand,   given any $\omega$ (up to a measure zero set of paths in $\Omega$), the  Birkhoff ergodic theorem implies that there exists a bi-infinite sequence of times $\{t_n\}_{n\in\mathbb{Z}}$ such that
\begin{equation}\label{eqn:t_n}
  \theta^{t_n}\omega \in \Omega_A,  \textrm{ for all } n \in\mathbb{Z}.
\end{equation}
See \cite{lu2011chaotic}. 
That is, for a.e. path $\omega$, the orbit of $\omega$ under the metric dynamical system $\theta^t:\Omega\to\Omega$ visits $\Omega_A$ infinitely often.

For any $A>0$ and $\omega\in\Omega$ fixed, define the set
\begin{equation}\label{eqn:Q_A_omega}
\begin{split}
Q_A(\omega)=&\{t\in\mathbb{R}\,|\, \theta^t \omega\in\Omega_A \}\\
=&\{t\in\mathbb{R}\,|\,|\omega(t+s) |\le  A+B|s|,\forall s\in\mathbb{R}\}.
\end{split}
\end{equation}

For any $T>0$, $A>0$ and $\omega\in\Omega$ fixed, we now define the following subset of $Q_{A}(\omega)$
\begin{equation}\label{eqn:Q_AT_omega}
\begin{split}
Q_{A,T}(\omega)=&\{t\in [-T,T]\,|\, \theta^t \omega \in\Omega_A \}\\
=&\{t\in [-T,T]\,|\,|\omega(t+s) |\le A+B|s|,\forall s\in\mathbb{R}\}.
\end{split}
\end{equation}
For $\delta>0$, and $A_\delta$ as in \eqref{eqn:Omega_A_delta}, we denote the   set in \eqref{eqn:Q_AT_omega} corresponding to $A_\delta$ by
$Q_{A_\delta,T}(\omega)$.

We  have the following monotonicity property:
\begin{equation}\label{eqn:Q_A_monotone}
  T_1\le T_2 \implies Q_{A,T_1}(\omega)\subseteq Q_{A,T_2}(\omega).
\end{equation}

We now recall \cite[Lemma 3.2]{yagasaki2018melnikov}.

\begin{lem}\label{lem:Yagasaki}
For any $\delta > 0$ there exists a random variable  $T_\delta(\omega)$
such that \begin{equation}
\label{eqn:Q_A_measure} T>T_\delta(\omega) \Rightarrow m(Q_{A_\delta,T}(\omega))>2(1-\delta)T,\end{equation}
where $m$ denotes the Lebesgue measure on $\mathbb{R}$.
\end{lem}\begin{proof}
Let $\chi_A$  be the characteristic function of a set $A$.
We have:
\[ m(\{{t}\in(0,T)\,|\, \theta^{t}\omega\in \Omega_{A_\delta} \})= \int_{0}^{T}\chi_{\Omega_{A_\delta}}(\theta^{s}\omega)ds.\]

Using the ergodicity and stationarity of $\eta(t)$, as well as \eqref{eqn:Omega_A_delta}, we obtain
\[\lim_{T\to\infty}\frac{1}{T}\int_{0}^{T}\chi_{\Omega_{A_\delta}}(\theta^{s}\omega)ds=
E[\chi_{\Omega_{A_\delta}}(\omega)]=\mathbb{P}(\Omega_{A_\delta})>1-\delta,\]

A similar result  holds when we take the limit as $T\to-\infty$.
Combining the two results concludes the proof.
\end{proof}

\subsection{Construction of a random bump function}
\label{sec:random_bump_function}
For each $t\in\mathbb{R}$,  and $\omega\in\Omega$ and $A>0$ we define the following sets
\begin{equation}\label{eqn:C_set}
  C_{A}(\omega)=\{s\in\mathbb{R}\,|\, |\omega(s)|\le A+B|s|\}
\end{equation}
which is a closed set in $\mathbb{R}$, and for $\rho>0$ small,
 \begin{equation}\label{eqn:U_set}
  U_{ A,\rho}(\omega)=\bigcup_{s\in C_{A}(\omega)} (s-\rho,s+\rho)
\end{equation}
which is a $\rho$-neighborhood of $C_{A}(\omega)$.

Let
 \begin{equation}\label{eqn:U_set}
 F_{ A,\rho}(\omega)= \mathbb{R}\setminus U_{ A,\rho}(\omega)
\end{equation}
which is a closed set in $\mathbb{R}$.  See Fig.~\ref{fig:Sets_2}.

\begin{figure}
\includegraphics[width=0.5\textwidth]{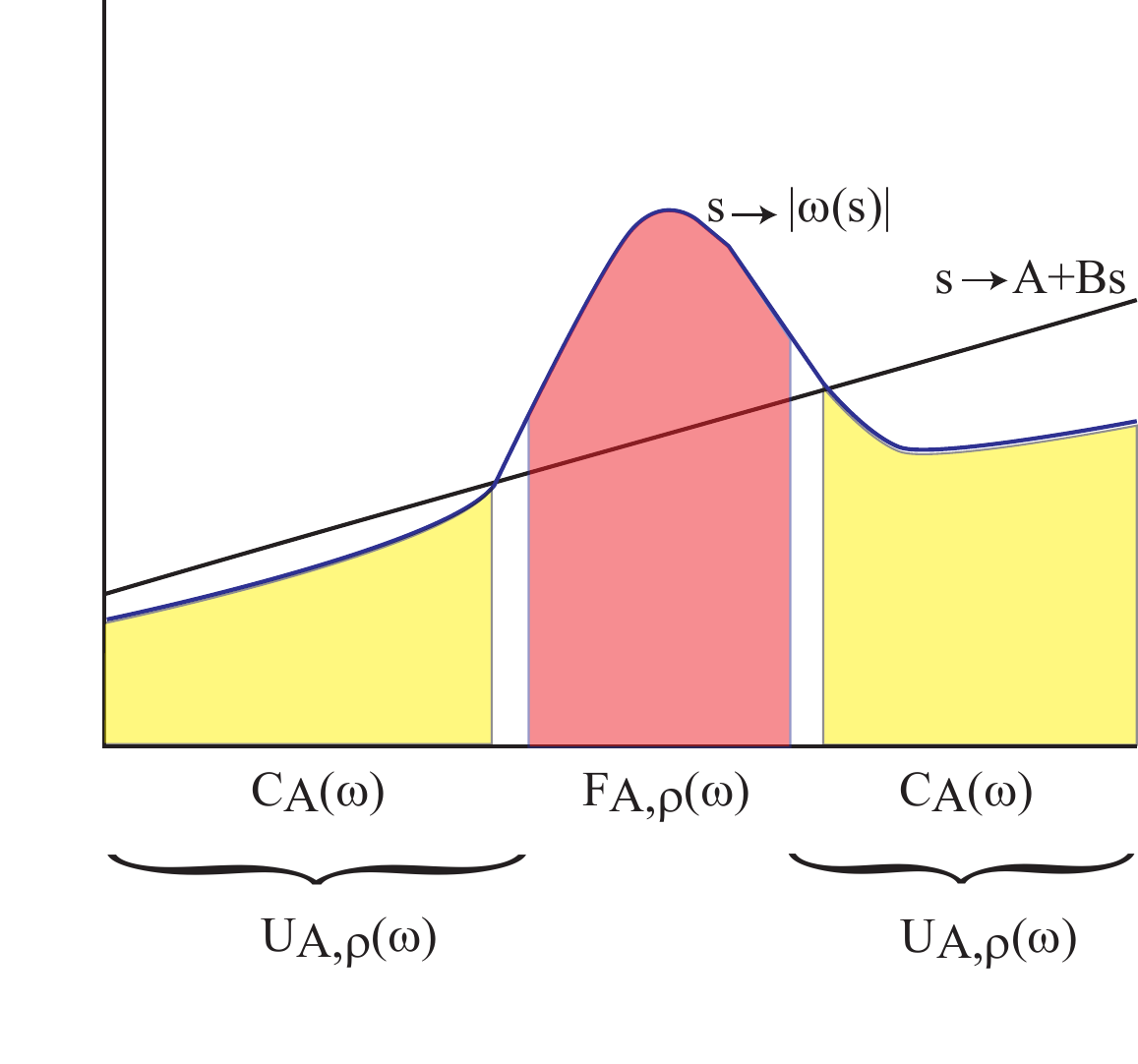}
\caption{Sets involved in the construction of a random bump function. }
\label{fig:Sets_2}
\end{figure}

It is clear that:
\begin{equation}\label{eqn:C_U_sets}
  C_{ A}(\omega)\subseteq U_{ A,\rho}(\omega) \textrm { and } d_H(C_{ A}(\omega), F_{ A,\rho}(\omega))=\rho,
\end{equation}
where $d_H$ refers to the Hausdorff distance.

If $s'\in \textrm{Cl}(U_{ A,\rho}(\omega))$, then there exists $s\in C_{A}(\omega)$ such that $|s-s'|\le \rho$, therefore,
by the H\"older property \eqref{eqn:Holder} of $\omega(t)$, we have
\begin{equation}\label{eqn:omega_U _minus C}
  |\omega(s)-\omega(s')|\le C_H|s-s'|^\alpha \le  C_H \rho^\alpha,
\end{equation}
hence
\begin{equation}\label{eqn:}
 \begin{split}
|\omega(s')|\le& |\omega(s')-\omega(s)|+|\omega(s)|\le C_H\cdot \rho^\alpha+A+B|s|\\\le&A'_{\rho}+B|s|,
 \end{split}
\end{equation}
where $A'_{\rho}:=C_H \cdot \rho^\alpha+A$.

Note that for $t\in\mathbb{R}$ we obviously have
\begin{equation}\label{eqn: C_U_theta_omega}
\begin{split}
  C_{A}(\theta^t\omega)=&\{s\in\mathbb{R}\,|\, |\theta^t \omega(s)|\le A+B|s|\}
\\=& \{s\in\mathbb{R}\,|\, |\omega(t+s)|\le A+Bs\},\\
\textrm{Cl}(U_{ A,\rho}(\omega))\subseteq&\{s \in\mathbb{R}\,|\, |\theta^t \omega(s)|\le A'_{\rho} +B|s|\}
\\ \subseteq& \{s\in\mathbb{R}\,|\, |\omega(t+s)|\le  A'_{\rho} +B|s|\}.
\end{split}
\end{equation}

We construct {\em  random bump function} $\psi_A:\mathbb{R}\to\mathbb{R}$ with $\psi_A(s,\omega)$  being a measurable function in $\omega$ for each $s$ fixed, and a $\C^\infty$ function in $s$   for each $\omega$ fixed,  such that
\begin{equation}\label{eqn:psi}
\begin{split}
 \psi_{A,\rho}(\cdot,\omega):\mathbb{R}\to[0,1],\\
  \psi_{A,\rho}(s,\omega)=\left\{
                           \begin{array}{ll}
                             1, & \hbox{for $s\in C_{A}(\omega)$,} \\
                             0, & \hbox{for $s\in F_{A,\rho}(\omega)$,}
                           \end{array}
                         \right.\\
\|D_s\psi_{A,\rho}(s,\omega)\|_{C^0}\le \frac{1}{\rho},
\end{split}
\end{equation}
where $D_s$ denotes the derivative with respect to $s$.

The latter condition comes from the fact that a slope of a line that changes from $0$ to $1$ within an interval of at least $\rho$ is at most
$\frac{1}{\rho}$. Therefore the upper bound on the first derivative of a bump function that is equal to $1$ on $C_{A}(\omega)$ and is supported on the closure of the $\rho$-neighborhood  $U_{ A,\omega}$ of $C_{A}(\omega)$ is at most $\frac{1}{\rho}$.  Recipes to construct such bump functions can be found in  \cite{nestruev2003smooth}.

Therefore,
\begin{equation}\label{eqn:psi_explicit}
\begin{split}
  \psi_{A,\rho}(s,\omega)=\left\{
                           \begin{array}{ll}
                             1, & \hbox{if $|\omega(s)|\le A+B|s|$,} \\
                             0, & \hbox{if  $|\omega(s)|\ge  A'_{\rho}+B|s|$,}
                           \end{array}
                         \right.
\end{split}
\end{equation}

Consequently, for $t\in\mathbb{R}$, we have
\begin{equation}\label{eqn:psi_shifted explicit}
\begin{split}
  \psi_{A,\rho}(s,\theta^t\omega)=\left\{
                           \begin{array}{ll}
                             1, & \hbox{if $|\omega(t+s)|\le A+B|s|$,} \\
                             0, & \hbox{if  $|\omega(t+s)|\ge  A'_{\rho}+B|s|$.}
                           \end{array}
                         \right.
\end{split}
\end{equation}

Note that if $t\in Q_{A}(\omega)$ then $|\omega(t+s)|\le A+B|s|$ for all $s$, hence $\psi_{A,\rho}(\cdot,\theta^t\omega)\equiv 1$.

\section{Geometric structures of  the unperturbed system}

\subsection{Coordinate system for the unperturbed  rotator-pendulum system}
\label{sec:coordinate_system}
We introduce a new coordinates system defined in a neighborhood of the homoclinic orbit $(p_0(t),q_0(t))$ of the pendulum
in the $(p,q)$-phase space, as we describe below.

Choose some fixed reference value $q_*\neq 0$ of the position coordinate $q$ of the pendulum.
The first coordinate of a point $(p,q)$  is
the  pendulum energy level $P(p,q)=\frac{1}{2}p^2+V(q)$ on which the point lies.
The second coordinate represents the time it takes for the solution $(p(\tau),q(\tau))$ to go from  $q_*$  to $(p,q)$ along the energy level $P$ corresponding to $(p,q)$. Note that for a given initial position $q_*$  the corresponding initial momentum $p_*$ is uniquely determined, up to sign,  by the energy condition $P(p_*,q_*)=P(p,q)$.
In order for the coordinate change $(p,q)\to (P,\tau)$ to be well defined,  we restrict to a  neighborhood $\mathcal{N}$  of the  homoclinic orbit $\{p_0(t),q_0(t)\}$ that does not contain any critical point of $P$, of the form
\[\mathcal{N}=\{(p,q)\,|\,  P_1<P<P_2, q_1<q<q_2\}\]
for some $P_1<0<P_2$ and $0<q_1<q_*<q_2< 1$.

The coordinate change $(p,q)\mapsto (P,\tau)$ is canonical, i.e., $dp\wedge dq=dP\wedge d\tau$ for $(p,q)\in\mathcal{N}$; see \cite{gidea2018global}, also \cite{gidea2021global}.
While the above coordinate change is only defined on $\mathcal{N}$,  the energy of the pendulum  $P$ as a function of $(p,q)$ is well defined at all points $(p,q)$.

For the rotator-pendulum system we  have the canonical coordinates $(I,\phi,P,\tau)$ for $(p,q)\in\mathcal{N}$.
In these coordinates the Hamiltonian $H_0$ is given by
\begin{equation}\label{eqn:H_0_I_theta_P_tau}
H_0(I,\phi,P,\tau)=h_0(I)+P.
\end{equation}

In Section \ref{sec:extended phase space} below, we will consider the extended phase space, by considering time $\t$ as an additional phase-space coordinate.
In this case we have the system of coordinates  $(I,\phi,P,\tau,\t)$ for $(p,q)\in\mathcal{N}$.

\subsection{Normally hyperbolic invariant manifold for the unperturbed  rotator-pendulum system}
\label{section:NHIM_ unperturbed rotator_pendulum}

Consider the unperturbed rotator-pendulum system given by $H_0$.

The point $(0,0)$ is a hyperbolic fixed point for the pendulum,  the  characteristic
exponents are $\beta=(-V''(0))^{1/2}>0$, $-\beta=-(-V''(0))^{1/2}<0$,
 and the corresponding unstable/stable eigenspaces are $E^\un=\textrm{Span}(v^u)$,
 $E^\st=\textrm{Span}(v^s)$, where
$v^\un=(-(-V''(0))^{1/2},1)$, $v^\st=((-V''(0))^{1/2},1)$.

Also, define
\begin{equation}\label{eqn:def_Eus} \begin{split} E^\un_z=& \{z\}\times \textrm{Span}(v^\un),\\
E^\st_z=& \{z\}\times \textrm{Span}(v^\st).\end{split} \end{equation}

It immediately follows that for  each closed interval $[a,b]\subseteq \R$, the set
\begin{equation}
\label{eqn:pendulm_NHIM}
\Lambda_0=\{(I,\phi, p,q)\,|\, I\in [a,b],\, p=q=0\}
\end{equation}
is a NHIM with boundary, where
the unstable and stable spaces $E^\un_z$ and
$E^\st_z$ at $z\in\Lambda_0$  are  given by \eqref{eqn:def_Eus}, respectively, and
the rates that appear in the definition of a NHIM are given by
$\beta$ for $E^\st$, $-\beta$ for $E^\un$,
and $\alpha=0$  for $T\Lambda_0$.

The stable and unstable manifolds of $\Lambda_0$ are denoted by $W^\st(\Lambda_0)$ and $W^\un(\Lambda_0)$,
respectively. They are $3$-dimensional manifolds, and $W^\st(\Lambda_0)=W^\un(\Lambda_0)$.
Relative to the  $(I,\phi,P,\tau)$ coordinates they can be locally written as graphs over the $(I,\phi,\tau)$ variables.
See, e.g.,  \cite{gidea2018global,gidea2021global}.

\subsection{Extended phase space}
\label{sec:extended phase space}
The system \eqref{eqn:hamiltonian_perturbed} is non-autonomous. We transform it into an autonomous system by making the time into an additional dependent variable (or additional phase space coordinate) $\t$, and denoting the independent variable by $t$:
\begin{equation}\label{eqn:hamiltonian_extended}
\begin{split}\frac{d}{dt}{z}
=& X^0(z)+\eps X^1(z,\eta(\t)),\\
\frac{d}{dt}{\t}=&1.
\end{split}
\end{equation}
We denote by $\tilde{\Phi}^t_\eps$ the flow for the (autonomous) extended system and $\tz=(z,\t)$.
The solution with $\tz(t_0)=(z_0,t_0)$ is given by
\begin{equation}\label{eqn:extended_flow_0}
  (z_\eps(t;z_0,t_0,\omega),t)=(\Phi^{t-t_0}_\eps(z_0,\theta^{t_0}\omega),t),
\end{equation}
where the flow $\Phi_\eps$ associated to  \eqref{eqn:hamiltonian_perturbed} is defined as in \eqref{eqn:RDE_RDS}.
It is easy to see that when $t=t_0$ we have $(z(t_0;z_0,t_0,\omega),t) =(\Phi^{0}_\eps(z_0,\theta^{t_0}\omega),t_0)=(z_0,t_0)$, as expected.

The extended flow of \eqref{eqn:hamiltonian_extended} is defined as
\begin{equation}\label{eqn:extended_flow}
\tilde\Phi^t_\eps(z_0,t_0, \omega)=(z(t;z_0,0,\omega),t+t_0)=(\Phi^t_\eps(z_0,\omega), t_0+t).
\end{equation}


The extended flow for the unperturbed system is
\[\tilde\Phi^t_0(z_0,t_0)=(z(t;z_0,0),t+t_0)=(\Phi^t_0(z_0), t_0+t).\]

For $\eps>0$, the expression of  the perturbed  flow $\tilde{\Phi}^t_\eps$ in the time-component  is the same as in the unperturbed case when $\eps=0$.


\section{Geometric structures of the perturbed system}

From now on, we will use the following notation convention: $t=$physical time, $\t=$time as an additional coordinate in the extended space, $s=$dummy variable.

For any given $\omega$, consider the set
\[Q_{A,T}(\omega)=\{t_0\in[-T,T]\,|\, |\theta^{t_0}\omega(s)|\le A+B|s|,\,\forall s\in\mathbb{R} \}.\]

If $t_0\in Q_{A,T}(\omega)$ then $\theta^{t_0}\omega\in \Omega_{A}$ and so $|\omega(t_0+s)|\le A+B|s|$ for all $s$.
Equivalently, $C_{A} (\theta^{t_0}\omega)=\mathbb{R}$.

If $t_0\in \mathbb{R}\setminus Q_{T,A}(\omega)$ we have the following possibilities:
\begin{itemize}
\item if $s\in C_{A}(\theta^{t_0}\omega)$ then $|\omega(t_0+s)|\le A+B|s|$;
\item  if $s\in \textrm{Cl}(U_{A,\rho}(\theta^{t_0} \omega ))$ then  $|\omega(t_0+s)|\le A'_\rho+B|s|$;
\item  if $|\omega(t_0+s)|\ge A'_\rho+B|s|$ then $s\in F_{A,\rho}(\theta^{t_0} \omega )$.
\end{itemize}

The corresponding  bump function \eqref{eqn:psi} is  $\psi_{A,\rho}(s,\theta^{t_0}\omega)$.  We have
\begin{equation}\label{eqn:}
  \psi_{A,\rho}(s,\theta^{t_0}\omega)=\left\{
                                    \begin{array}{ll}
                                      1, & \hbox{for $s\in C_{A}(\theta^{t_0}\omega)$;} \\
                                      0, & \hbox{for $s\in F_{A,\rho}(\theta^{t_0} \omega )$.}
                                    \end{array}
                                  \right.
\end{equation}

In particular, for $t_0=0$ we have
\begin{equation}\label{eqn:}
  \psi_{A,\rho}(s, \omega)=\left\{
                                    \begin{array}{ll}
                                      1, & \hbox{for $s\in C_{A}(\omega)$;} \\
                                      0, & \hbox{for $s\in F_{A,\rho}(\omega )$.}
                                    \end{array}
                                  \right.
\end{equation}

Now, we modify the vector field $X^1_\eta(z,\t)$ by multiplying the Hamiltonian function by the random bump function $(\t,\omega)\mapsto \psi_{A}(\t,\omega)$ defined on the extended space:
\begin{equation}\label{eqn:m}
  \hat X^1_\omega(z,\t)=J\nabla \left(\psi_{A,\rho}(\t,\omega)H_1(z)\right)\omega(\t) .
\end{equation}

(Note that in the notation for the bump function we switched  from the dummy variable $s$ to the time-coordinate~$\t$.)

We have
\begin{equation}\label{eqn:X_1 hat}
  \begin{split}
 \hat X^1_\omega(z,\t)= \left\{
                         \begin{array}{ll}
                            X^1_\omega(z,\t), & \hbox{for $\t\in C_{A}(\omega)$,} \\
                           0, & \hbox{for $\t\in  F_{A,\rho}(\omega )$.}
                         \end{array}
                       \right.
\end{split}
\end{equation}

The modified system is
\begin{equation}\label{eqn:z_hat}
\begin{split}
\frac{d}{dt}\hat z _\eps=&J\nabla H_0(\hat z_\eps)+J\nabla \left(\psi_{A}(\t,\omega)H_1(\hz_\eps)\right)\omega(\t) , \\
\frac{d}{dt}\t=&1 .
\end{split}
\end{equation}

%

For fixed $\omega$ and $t_0\in \mathbb{R}$, by \eqref{eqn:Caratheodory} and \eqref{eqn:RDE_RDS_2}, the solution of  \eqref{eqn:z_hat}  satisfies
\begin{equation}\label{eqn:z_hat_integral}
\begin{split}
\hz_\eps&(t;z_0,t_0,\omega)= \hz_\eps(t-t_0;z_0,0,\theta^{t_0}\omega)= \pi_z[\hat\Phi^{t-t_0}_\eps(z_0,0,\theta^{t_0}\omega)]\\&=z_0+\int_{0}^{t-t_0}\left[ J\nabla H_0(\hz_\eps(s))+J\nabla \left(\psi_{A,\rho}(s,\theta^{t_0}\omega)H_1(\hz_\eps(s))\right)\theta^{t_0}\omega(s)\right] ds.
\end{split}
\end{equation}

If we denote the elapsed time $t-t_0=t'$, we have
\begin{equation}\label{eqn:z_hat_integral_prime}
\begin{split}
\pi_z[\hat\Phi^{t'}_\eps(z_0,0,\theta^{t_0}\omega)]=z_0+\int_{0}^{t'}&\left[ J\nabla H_0(\hz_\eps(s))\right.\\
&\left.+J\nabla \left(\psi_{A,\rho}(s,\theta^{t_0}\omega)H_1(\hz_\eps(s))\right)\omega(t_0+s)\right] ds.
\end{split}
\end{equation}

For $t_0\in  Q_{A,T}(\omega)$, the solution $\hat\Phi^{t'}_\eps(z,\theta^{t_0}\omega)$ of the modified system \eqref{eqn:z_hat} coincides with the solution   $\tilde\Phi^{t'}_\eps(z,\theta^{t_0}\omega)$ of the original system
\eqref{eqn:hamiltonian_extended}.

\subsection{Persistence of the NHIM}
\label{sec:persistence_NHIM}
The result below will be used to prove Theorem \ref{thm:random_NHIM}.
The first part of the result says that the flow of the modified  perturbed system is close to the flow of the unperturbed system.
The second part of the result says that for the distinguished set of times, the  modified  perturbed flow coincides with the original  perturbed flow.

\begin{prop}\label{prop:perturbed_flow}
Assume that the system \eqref{eqn:hamiltonian_perturbed} satisfies (P-i), (P-ii), (R-i), (R-ii), (R-iii) (but not necessarily (H1)).
Fix $\delta>0$, $\rho>0$, and $k\in(0,1)$.

Then, there exists  $\eps_0>0$ such that,  for   every $\eps$ with $0<\eps<\eps_0$, all $(z_0,t_0)$, and  a.e. $\omega \in\Omega$, for the modified system we have
\begin{equation}\label{eqn:perturbed_flow}
\|\hat{\Phi}^t_\eps(z_0,t_0,\omega)-\hat{\Phi}^t_0(z_0,t_0)\|_{C^1}\leq   C\eps^{1-c} \textrm { for } t\in [0,1].
\end{equation}

In particular,  there exist a positive random variable $T_\delta(\omega)$, a closed set  $Q_{A_\delta,T_\delta}(\omega)\subseteq [-T_\delta, T_\delta]$ satisfying \eqref{eqn:Q_A_measure} and \eqref{eqn:Q_A_monotone},
and  $\eps_0>0$ such that,  for   every $\eps$ with $0<\eps<\eps_0$, all $(z_0,t_0)$,   a.e. $\omega \in\Omega$, and  every $t_0\in Q_{A_\delta,T_\delta}(\omega)$,   the solution of the modified system and the solution for the extended system coincide, i.e.,
\begin{equation}\label{eqn:modified original}
 \hat{\Phi}^t_\eps(z_0,t_0,\omega)=\tilde{\Phi}^t_\eps(z_0,t_0,\omega) \textrm{ for all } t\in\mathbb{R},
\end{equation}
and therefore
\begin{equation}\label{eqn:perturbed_flow_2}
  \|\tilde{\Phi}^t_\eps(z_0,t_0,\omega)-\tilde{\Phi}^t_0(z_0,t_0)\|_{C^1}\leq   C\eps^{1-c} \textrm { for } t\in [0,1].
 \end{equation}
\end{prop}
\begin{proof}


Consider the set $Q_{A_\delta,T_\delta}(\omega)$.

To prove \eqref{eqn:perturbed_flow} it is sufficient to show that the modified, perturbed flow and the unperturbed flow, when we shift the origin of time at $t_0$,  are $\C^1$-close, that is
\[  \|\pi_z[\hat\Phi^{t-t_0}_\eps(z_0,0,\theta^{t_0}\omega)]-\pi_z[\tilde\Phi^{t-t_0}_0(z_0,0)]\| _{C^1}\leq   C\eps^{1-c} \textrm { for } t\in [0,1].\]

To simplify notation, we substitute $t-t_0\mapsto t$ and write
\[ \hz_\eps(t)=\pi_z\left[ \hat\Phi^{t}_\eps(z_0,0,\theta^{t_0}\omega)\right] \textrm { and } z_0(t)= \pi_z\left[ \tilde\Phi^{t}_0(z_0,0) \right].\]

We write the solution $z_0(t)$ of the unperturbed system  and the solution  $\hz_\eps(t)$ of
of the perturbed, modified system \eqref{eqn:z_hat}  in integral form as in \eqref{eqn:z_hat_integral_prime} (with $t'$ replaced by $t$)
\begin{equation}\label{eqn:pert1}
  \begin{split}
  z_0(t)=&z_0+\int_{0}^t J\nabla H_0(z_0(s))ds \\
  \hz_\eps(t)=&z_0+\int_{0}^t [J\nabla H_0(\hz_\eps(s))+\eps J\nabla (\psi_{A_{\delta},\rho} (s,\theta^{t_0}\omega)H_1(\hz_\eps(s)))\omega(t_0+s)]ds .
  \end{split}
\end{equation}
By subtraction we obtain
\begin{equation}\label{eqn:pert2}
\begin{split}
   \|\hz_\eps(t)-z_0(t)\|\leq& \int_{0}^{t} |J\nabla H_0(\hz_\eps(s))-J\nabla H_0(z_0(s))|ds\\
&+\eps \int_{0}^{t} |J\nabla (\psi_{A_{\delta},\rho} (s,\theta^{t_0}\omega)H_1(\hz_\eps(s)))| |\theta^{t_0}\omega(s)|ds .
\end{split}
\end{equation}
Restricting $z$ to some suitable, compact domain, we let $K_1$ be the Lipschitz constant for $J\nabla H_0$,
and $K_2$ be the supremum of $\|J\nabla (\psi_{A_\delta,\rho} H_1)\|$  (recall that  $H_1$ is uniformly $\C^2$
and $0\le _{A_\delta,\rho}\le 1$). Since $\psi_{A_\delta,\rho} (s,\theta^{t_0})=0$ whenever $|\omega(t_0+s)|>A'_{\delta,\rho}+B|s|$, we then have
\[|J\nabla (\psi_{A_{\delta},\rho} (s,\theta^{t_0}\omega) H_1(\hz_\eps(s)))| |\omega(t_0+s)| \le K_2 (A'_{\delta,\rho}+B s ),\]
where we denote $A'_{\delta,\rho}:= C_H\cdot\rho^\alpha+A_\delta$.

From \eqref{eqn:pert2} we infer
\begin{equation}\label{eqn:pert3}
\begin{split}
   \|\hz_\eps(t)-z_0(t)\|\leq & K_1 \int_{0}^{t} \| \hz_\eps(s)- z_0(s) \|ds\\&+\eps K_2\int_{0}^{t} ( A'_{\delta,\rho}+B s)ds .
\end{split}
\end{equation}

Hence \eqref{eqn:pert3} yields
\begin{equation}\label{eqn:pert4}
\begin{split}
   \|\hz_\eps(t)-z_0(t)\|\leq& K_1 \int_{0}^{t} \| \hz_\eps(s)- z_0(s) \|ds +\eps K_2  \int_{0}^{t} (A'_{\delta,\rho} +B s) ds\\
   = &K_1 \int_{0}^{t} \| \hz_\eps(s)- z_0(s) \|ds+\eps K_2   \left(A'_{\delta,\rho}  t +\frac{B}{2}t ^2\right).
   \end{split}
\end{equation}

Applying Gronwall's Inequality -- I \ref{eqn:gronwall-I} for $\delta_0=0$, $\delta_1=\eps K_2 A'_{\delta,\rho}$, $\delta_2=\eps K_2 B/2$ and $\delta_3=K_1$, we obtain

\begin{equation}\label{eqn:pert5}
\begin{split}
   \|\hz_\eps(t)-z_0(t)\|\leq&  \left( \delta_0+\delta_1 t+ \delta_2 t^2 \right) e^{\delta_3 t}\\
   =&\eps \left(  K_2 A'_{\delta,\rho} t+ \frac{K_2 B}{2} t^2 \right) e^{K_1 t}\\
   =& \eps \left(  \bar A_{\delta,\rho}  t+ \bar B t^2 \right) e^{K_1 t},
   \end{split}
\end{equation}
where $\bar A_{\delta,\rho}=K_2 A'_{\delta,\rho}$ and $\bar B=\frac{K_2 B}{2}$.

Fix    $0<k<1$ and $0<k'<1-k$. For $0\leq t\leq \frac{k}{K_1}\ln\left( \frac{1}{\eps}\right)=\frac{1}{K_1}\ln\left( \left(\frac{1}{\eps}\right)^k\right)$ we have
\begin{equation}\label{eqn:pert6}
\begin{split}
   \|\hz_\eps(t)-z_0(t)\|\leq&   \eps\left(\bar A_{\delta,\rho}\frac{k}{K_1}\ln(\frac{1}{\eps})+ \bar B\frac{k^2}{K^2_1}\left(\ln(\frac{1}{\eps})\right)^2 \right) \left(\frac{1}{\eps}\right)^k\\
   =&\eps^{1-k-k'}\cdot \eps^{k'} \left(\bar A_{\delta,\rho}\frac{k}{K_1}\ln(\frac{1}{\eps})+ \bar B\frac{k^2}{K^2_1}\left(\ln(\frac{1}{\eps})\right)^2 \right).
   \end{split}
\end{equation}

\begin{equation}\label{eqn:eps_ln}
  \lim_{\eps\to 0}\eps^{k'}\ln(\frac{1}{\eps})=0 \textrm { and } \lim_{\eps\to 0}\eps^{k'}\left(\ln(\frac{1}{\eps})\right)^2=0.
\end{equation}

From \eqref{eqn:eps_ln} and the fact $\frac{k}{K_1}\ln\left( \frac{1}{\eps}\right) \to \infty $ as $\eps\to 0$, there exist $\eps_0  >0$ and $C^0_{\delta,\rho}>0$  and such that for all $0<\eps<\eps_0$ we have
\[ [0,1]\subset\left [0, ({k}/{K_1})\ln\left(  {1}/{\eps}\right)\right],\]
and
\[\eps^{k'} \left(\bar A_{\delta,\rho} \frac{k}{K_1}\ln(\frac{1}{\eps})+ \bar B\frac{k^2}{K^2_1}\left(\ln(\frac{1}{\eps})\right)^2 \right)<C^0_{\delta,\rho}.\]

Therefore, denoting $c_0=k+k'$, we obtain

\begin{equation}\label{eqn:C1}
  \|\hz_\eps(t)-z_0(t)\|\leq C^0_{\delta,\rho}\eps^{1-c_0}, \textrm { for } t\in[0,1].
\end{equation}

We have only showed that $\hz_\eps(t)$ and $z_0(t)$ are $C^0$-close. Now we make a similar argument to show that $\hz_\eps(t)$ and $z_0(t)$ are $\C^1$-close.

Denote
\begin{equation}\label{eqn:pert7}
    \begin{split}
   \frac{d}{dz}\pi_z\left[\tilde\Phi^t_0(z_0,0)\right]:=&\xi_0(t)\\
   \frac{d}{dz}\pi_z\left[\hat\Phi^{t}_\eps(z_0,0,\theta^{t_0}\omega)\right]:=&\hat\xi_\eps(t).
   \end{split}
\end{equation}

Then $\xi_0(z)$ and $\xi_\eps(z)$  satisfy the variational equations
\begin{equation}\label{eqn:pert8}
    \begin{split}
    \dot{\xi}_0(t)=&DJ\nabla H_0(z_0(t))\xi_0(t)\\
  \dot{\hat\xi}_\eps(t) =&DJ\nabla H_0(z_\eps(t))\xi_\eps(t)+DJ\nabla\left(\psi_{A}(t,;\theta^{t_0}\omega)H_1(\hz_\eps(t))\right)\xi_\eps(t)\omega(t),
   \end{split}
\end{equation}
where the derivation $D$ is with respect to $z$.

Then
\begin{equation}\label{eqn:pert9}
\begin{split}
   \|\hat\xi_\eps(t)-\xi_0(t)\|\leq& \int_{0}^{t} \|DJ\nabla H_0(z_\eps(s))\hat\xi_\eps(s)-DJ\nabla H_0(z_0(t))\hat\xi_0(s)\|ds\\
   &+\eps \int_{0}^{t} \|DJ\nabla\left(\psi_{A_{\delta},\rho}(s;\theta^{t_0}\omega)H_1(\hz_\eps(s))\right)\hat\xi_\eps(s)\| |\omega(t_0+s)|ds.
\end{split}
\end{equation}

Restricting  $z$ to some suitable, compact domain, let $K'_1$ be the Lipschitz constant for $DJ\nabla H_0$,  $K'_2$ such that $\|DJ\nabla \left(\psi_{A_{\delta},\rho}(s;\theta^{t_0}\omega)H_1(\hz_\eps(s))\right)\|<K'_2$ (recall that $H_1$ is uniformly $\C^2$ and $\psi_{A_{\delta},\rho}$ is uniformly $\C^1$), and $K'_3>0$ such that $\|\xi_0(t)\|<K'_3$.
Therefore:
\begin{equation}\label{eqn:pert10}
\begin{split}
   \|\hat\xi_\eps(t)-\xi_0(t)\|\leq& K'_1\int_{0}^{t} \|\hat\xi_\eps(s)-\xi_0(s)\|ds\\
   &+\eps K'_2\int_{0}^{t} \|\hat\xi_\eps(s)\| |\omega(t_0+s)|ds\\
   \leq& K'_1\int_{0}^{t} \|\hat\xi_\eps(s)-\xi_0(s)\|ds\\
   &+\eps K'_2\int_{0}^{t} \|\hat\xi_\eps(s)-\xi_0(s)\| |\omega(t_0+s)|ds\\&+\eps K'_2\int_{t_0}^{t} \|\xi_0(s)\| |\omega(t_0+s)|ds\\
   \leq& K'_1\int_{0}^{t} \|\hat\xi_\eps(s)-\xi_0(s)\|ds\\
   &+\eps K'_2\int_{0}^{t} \|\hat\xi_\eps(s)-\xi_0(s)\| (A'_{\delta,\rho}+Bs)ds\\&+\eps K'_2 \int_{t_0}^{t} \|\hat\xi_0(s)\| (A'_{\delta,\rho}+Bs) ds\\
   \leq& \int_{0}^{t} (K'_1+\eps K'_2A'_{\delta,\rho} +\eps K'_2 B s)\|\hat\xi_\eps(s)-\xi_0(s)\|ds \\&+\eps K'_2K'_3  \left(A'_{\delta,\rho} t+\frac{B}{2}t^2\right)
\end{split}
\end{equation}

Applying Gronwall's Inequality -- II \ref{eqn:gronwall-II} for $\delta_0=0$, $\delta_1=\eps K'_2K'_3 A'_{\delta,\rho}$, $\delta_2=\eps\frac{K'_2K'_3 B}{2}$,
$\delta_3=K'_1+\eps K'_2A'_{\delta,\rho}$ and $\delta_4=\eps K'_2B$, we obtain
\begin{equation}\label{eqn:pert11}
  \begin{split}
     \|\hat\xi_\eps(t)-\xi_0(t)\|\leq & \eps\left ( K'_2K'_3 A'_{\delta,\rho} t+\frac{K'_2K'_3 B}{2}t^2\right )e^{\left[( K'_1+\eps K'_2A'_{\delta,\rho})t+\eps\frac{K'_2K'_3 B}{2}t^2\right]}  \\
       = &\eps\left(At+Bt^2\right)\eps^{\left[Ct+\eps(Dt+Et^2)\right]}.
  \end{split}
\end{equation}
where $\bar A_{\delta,\rho} =K'_2K'_3 A'_{\delta,\rho}$, $\bar B=\frac{K'_2K'_3 B}{2}$, $\bar C=K'_1$,   $\bar D_{\delta,\rho}=K'_2 A'_{\delta,\rho}$, and $\bar E=\frac{K'_2K'_3 B}{2}$.

Fix    $0<k<1$ and $0<k'<1-k$.
For $0\leq t \leq \frac{k}{\bar C}\ln(\frac{1}{\eps})$, we have
\begin{equation}\label{eqn:per12}
 \eps(\bar D_{\delta,\rho} t+\bar Et^2)\leq \eps \left(\bar D_{\delta,\rho}\frac{k}{\bar C}\ln(\frac{1}{\eps})+\bar E\frac{k^2}{\bar C^2 }\left(\ln(\frac{1}{\eps})\right)^2\right)\to 0 \textrm{ as }\eps\to 0
\end{equation}
due to \eqref{eqn:eps_ln}.
Therefore, there exists $\bar F_{\delta,\rho} >0$ such that, if $0<\eps<\eps_0$, for sufficiently small $\eps_0$, we obtain
\begin{equation}\label{eqn:per13}
 e^{\eps(\bar D_{\delta,\rho} t+\bar Et^2)}\leq \bar F_{\delta,\rho} .
 \end{equation}

Using again \eqref{eqn:eps_ln}, we obtain that there exist $\eps_0>0$, $C^1_{\delta,\rho}>0$ such that for $0<\eps<\eps_0$
\begin{equation}\label{eqn:per14}
 \begin{split}
 \|\hat\xi_\eps(t)-\xi_0(t)\|\leq &\eps^{1-k}\left(\bar A_{\delta,\rho}\frac{k}{K_1}\ln(\frac{1}{\eps})+ \bar B\frac{k^2}{K^2_1}\left(\ln(\frac{1}{\eps})\right)^2 \right)\bar F_\delta\\
 =&\eps^{1-k-k'}\cdot\eps^{k'}\left(\bar A_{\delta,\rho}\frac{k}{K_1}\ln(\frac{1}{\eps})+ \bar B\frac{k^2}{K^2_1}\left(\ln(\frac{1}{\eps})\right)^2 \right)\bar F_{\delta,\rho} \\
 \leq& C^1_{\delta,\rho}\eps^{1-c_1},
 \end{split}
 \end{equation}
 where $c_1=k+k'$.

 We obtain
\begin{equation}\label{eqn:C2}
  \|\hat\xi_\eps(t)-\xi_0(t)\|\leq C^1_{\delta,\rho}\eps^{1-c_1}, \textrm { for } t\in[0,1].
\end{equation}

Combining \eqref{eqn:C1} and \eqref{eqn:C2} we obtain

\begin{equation}\label{eqn:C3}
  \|\hz_\eps(t)-z_0(t)\|_{\C^1}\leq C_{\delta,\rho} \eps^{1-c}, \textrm { for } t\in[0,1],
\end{equation}
where $C_{\delta,\rho} =\max\{C^0_{\delta,\rho} ,C^1_{\delta,\rho} \}$ and $c=\max \{c_0,c_1\}$.
Since $k,k'$ are arbitrary, we can choose in fact any $c\in(0,1)$.
\end{proof}

%
%

We can use Proposition \ref{prop:perturbed_flow} to prove the first main result of the paper.

\begin{proof}[\bf Proof of Theorem \ref{thm:random_NHIM}]

Let
\[\tilde\Lambda_0=\{(I,\phi,p,q,t_0)\,|\,I\in[a,b],\,\phi\in\mathbb{T}^1,\, p=q=0, t_0\in\mathbb{R} \} \]
be the NHIM of the extended, unperturbed system.

Consider the perturbed, modified  system given by \eqref{eqn:z_hat}.

Proposition \ref{prop:perturbed_flow} implies that the flow of the perturbed,  modified system and the flow  of the unperturbed systems are $C\eps^{1-c}$ close
to one another in $\C^1$, for all $\eps$ smaller than some $\eps_0$. We choose $\eps_0$ small enough so that
\[C\eps_0^{1-c}<\bar{\eps},\]
where $\bar{\eps}$ is the smallness parameter that appears in \eqref{eqn: flow_error}, in the statement of Theorem \ref{thm:Bates1}.

Applying  Theorem \ref{thm:Bates1}, we obtain the existence of the normally hyperbolic manifold $\hat\Lambda_\eps(\omega)$ and its stable and unstable manifolds, satisfying the desired properties.
We note that, although $\tilde\Lambda_0$ is not compact, the result on the persistence of the NHIM still applies since the perturbation has uniformly bounded derivatives (see \cite{HirschPS77}).

By Lemma \ref{lem:Yagasaki}, for  a.e. $\omega\in \Omega$, the set $Q_{A_\delta,T_\delta}(\omega)$, consisting of the times ${t_0}$ for which $\theta^{t_0}(\omega)\in\Omega_{A_\delta}$ has measure at least $2(1-\delta)T_\delta$.

For $t_0\in \Omega_{A_\delta,T}$, the solution of  \eqref{eqn:z_hat} with $(\hat z_\eps(t_0),\t(t_0))=(z_0,t_0)$ coincides with the solution of
\eqref{eqn:hamiltonian_extended} with the same initial condition.
Thus, for $t_0\in Q_{A_\delta,T_\delta}(\omega)$, we obtain that the normally hyperbolic manifold $\hat\Lambda_\eps(\theta^{t_0}(\omega))$ for the modified system represents a normally hyperbolic manifold $ \Lambda_\eps(\theta^{t_0}(\omega))$ for the original system.
The same statement holds for  its stable and unstable manifolds, which satisfy  the desired properties.

We have noted that the set $\Omega_{A_\delta,T_\delta}\subseteq \Omega$ is not closed under $\theta^t$, therefore the equivariance property of the NHIM and on its stable and unstable manifolds is restricted to those $t_0,t_1\in\mathbb{R}$ such that $t_0,t_0+t_1\in Q_{A_\delta,T_\delta}(\omega)$.
Even though the normally hyperbolic manifold exist for those paths
$\theta^{t_0}\omega$ which are in $Q_{A_\delta,T_\delta}(\omega)$,
the initial path $\omega$ is arbitrary. Starting with an arbitrary path $\omega$, there exits a large measure set of $t_0$ for which $\theta^{t_0}\omega\in Q_{A_\delta,T_\delta}(\omega)$.
\end{proof}

%

\section{Existence of transverse homoclinic intersections}
\label{sec:transverse}

\subsection{Distance between stable and unstable manifolds}
The unperturbed stable and
unstable manifolds, $W^\st(\tLambda_0)$ and $W^\un(\tLambda_0)$ in the extended space
coincide along the homoclinic manifold, which is given in the coordinates $(I,\phi,P,\tau,\t)$ defined  in Section \ref{sec:coordinate_system}, by
\[\{(I, \phi,P,\tau,\t)\,|\, P=0\}.\]


Define the section:
\[ \Sigma_{t_0}= \{ (I, \phi,P,\tau,\t)\,|\, \t=t_0 \}.\]

For $t_0\in Q_{A_\delta,T_\delta}(\omega)$ and $\eps_0$ sufficiently enough,   the perturbed invariant manifold  for the original system $W^\st
(\Lambda_\eps(\theta^{t_0}\omega))$, $W^\un(\Lambda_\eps(\theta^{t_0}\omega))$ exist in $\Sigma_{t_0}$,
 and they are  $C^1$-close to $W^\st(\Lambda_0)$, $W^\un(\Lambda_0)$,  respectively. Moreover, the invariant manifold $W^\st(\Lambda_\eps(\theta^{t_0}\omega))$ can be represented as a graph $P^\st=P^\st(I,\phi,\tau,t_0)$ over the variables $(I,\phi,\tau)$, and, similarly, the invariant manifold $W^\un(\tLambda_\eps(\theta^{t_0}\omega))$ can be written as a graph $P^\un=P^\un(I,\phi,\tau,t_0)$, where $\t=t_0$ is fixed.
 See Fig.~\ref{fig:Wu_Ws_random}.

The next result says that we can express the distance between the perturbed stable and unstable manifolds as a Melnikov-type integral.
We will use this to find crossings of the stable and unstable manifolds as   zeros of the Melnikov integral.

\begin{figure}
\includegraphics[width=0.6\textwidth]{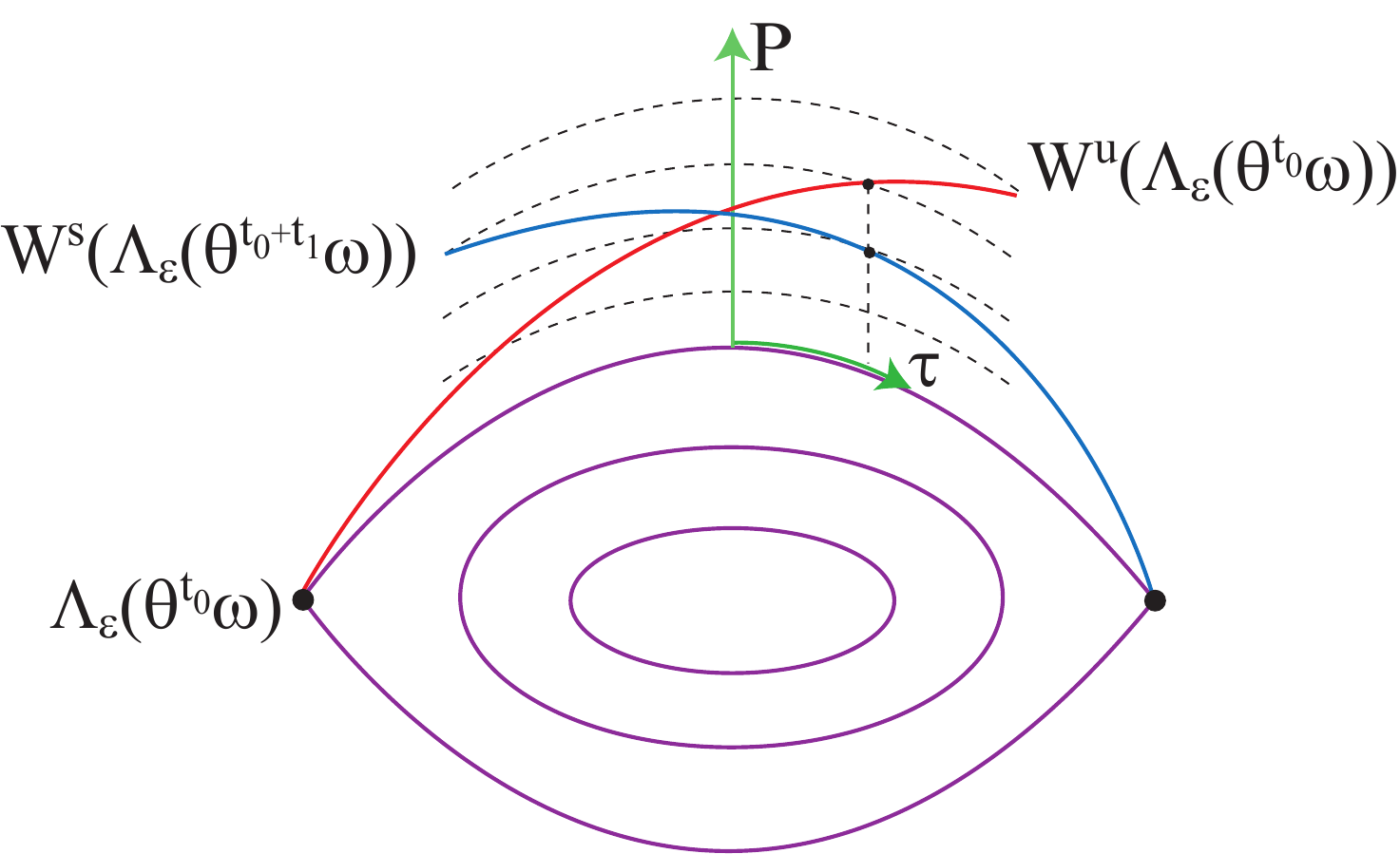}
\caption{Stable and unstable manifolds as graphs.}
\label{fig:Wu_Ws_random}
\end{figure}

\begin{prop}\label{thm:transverse_homoclinic}
Let $t_0\in Q_{A_\delta,T_\delta}(\omega)$ and $\eps_0$ sufficiently small. Consider a point $\tz^\st_\eps \in W^\st(\Lambda_\eps(\theta^{t_0}\omega))$ given by $P^\st(\tz^\st_\eps)=P^\st(I,\phi,\tau,t_0)$, and a point
$\tz^\un_\eps \in W^\un(\Lambda_\eps(\theta^{t_0}\omega))$ given by $P^\un(\tz^\un_\eps)=P^\un(I,\phi,\tau,t_0)$, for the same coordinates $(I,\phi,\tau,t_0)$.

Then
\begin{equation}\label{eqn:Pun_minus_Pst}\begin{split}
P&(\tz^\st_\eps) -P(\tz^\un_\eps) \\
& =-\eps\int_{-\infty}^{\infty}  \{P,H_1\}(I,\phi+\nu(I) s,p_0(\tau+s),q_0(\tau+s),t_0+s) \omega(t_0+s)  \, ds\\
&\quad+O(\eps^{1+\rho})
\end{split}
\end{equation}
for some $\rho\in(0,1)$.
\end{prop}

\begin{proof}
Suppose that $\tz_0=(I,\phi,p_0(\tau),q_0(\tau),t_0)$ is a homoclinic point for $\tPhi^t_0$.
Then the stable and unstable foot-points are both given by $\tz_0^\pm =(I,\phi,0,0,t_0)$, as
the stable foot-point and the unstable foot-point coincide in the unperturbed case.
Hence, in \eqref{eqn:Pun_minus_Pst},  $\left(I,\phi+\nu(I) s,p_0(\tau+s),q_0(\tau+t),t_0+s \right)$
 represents the effect of the unperturbed flow  ${\tPhi}^s_0$ on the homoclinic point $\tz_0$, and
$\left(I,\phi+\nu(I) s,0,0,t_0+s \right)$ represents the effect of the unperturbed flow ${\tPhi}^s_0$ on the foot-point $\tz_0^\pm$.

Since $d(\tPhi^s(\tz_0), \tPhi^s(\tz^\pm_0)\to 0$ exponentially fast as $s\to\pm\infty$,
\begin{equation*}\begin{split}\{P,H_1\}&\left(I,\phi+\nu(I) s,p(\tau+t),q(\tau+s),t_0+s\right) \\
&\, -\{P,H_1\}\left(I,\phi+\nu(I) s,0,0,t_0+s \right)  \to 0\end{split}\end{equation*}
exponentially fast   as $t\to\infty$.
Note that $\{P,H_1\}=V'(q)\frac{\partial H_1}{\partial p}-p\frac{\partial H_1}{\partial q}$ vanishes at $p=q=0$.
Since $t_0\in Q_{A_\delta,T_\delta}$,  $|\omega(t_0+s)|<A_\delta+B|s|$ for all $s$, so we obtain
\[\{P,H_1\}\left(I,\phi+\nu(I)t,p(\tau+s),q(\tau+s)\right)\omega(t_0+s)\to 0\] exponentially fast   as $s\to\infty$.
Thus, the improper integral in \eqref{eqn:Pun_minus_Pst} is absolutely convergent.

For $\tz^\st_\eps\in W^\st_\eps(\Lambda_\eps(\theta^{t_0}\omega))$  let $\Omega^\st_\eps({\tz}^\st_\eps)$ be the foot-point of the unique stable fiber of   $W^\st(\Lambda_\eps(\theta^{t_0}\omega))$ through ${\tz}^\st_\eps$ .
Using the fundamental theorem of calculus and
\begin{equation}\label{FTC_0}
\begin{split}
P({\tz}^\st_\eps) - &P({\Omega}^\st_\eps({\tz}^\st_\eps))
= P( {\tPhi}^T_\eps({\tz}^\st_\eps)) - P({\tPhi}^T_\eps {\Omega}_\eps^\st({\tz}^\st_\eps) ) \\
& -\int_0^T  \frac{d}{d s}\big[ P( {\tPhi}^s_\eps({\tz}^\st_\eps)) -
P( {\tPhi}^s_\eps {\Omega}_\eps^\st({\tz}^\st_\eps) ) \big] \, d s\\
= & P( {\tPhi}^T_\eps({\tz}^\st_\eps)) - P({\tPhi}^T_\eps {\Omega}_\eps^\st({\tz}^\st_\eps) ) \\
& -\eps\int_0^T \big[ (J\nabla H_1)P( {\tPhi}^s_\eps({\tz}^\st_\eps))\ -
(J\nabla H_1)P( {\tPhi}^s_\eps {\omega}_\eps^\st({\tz}^\st_\eps) )\big]  \omega(t_0+s) \, d s \\
\end{split}
\end{equation}

The vector field $(J\nabla H_1)$ is thought of as derivation, and so  $(J\nabla H_1)P$ is the corresponding directional derivative of $P$.
Hence \begin{equation*}\begin{split}
(J\nabla H_1)P({\tPhi}^s_\eps({\tz}))=&-\eps \left[\frac{\partial P}{\partial p}\frac{\partial H_1}{\partial q}+
\frac{\partial P}{\partial q}\frac{\partial H_1  }{\partial p}\right]({\tPhi}^s_\eps({\tz}))\omega(t_0+s)\\
=&\eps \{P,H_1\}({\tPhi}^s_\eps({\tz}))\omega(t_0+s),
\end{split}
\end{equation*}
where $\{\cdot,\cdot\}$ denotes the Poisson bracket.

Letting $T\to+\infty$, since $\left[P( {\tPhi}^T_\eps({\tz}^\st_\eps)) - P({\tPhi}^T_\eps {\Omega}_\eps^\st({\tz}^\st_\eps) )\right]\omega(t_0+s)\to 0$, we obtain
\begin{equation}\label{FTC_1}
\begin{split}
P({\tz}^\st_\eps) - P({\Omega}^\st_\eps({\tz}^\st_\eps))
=
& -\eps\int_0^{+\infty} \big[ \{P,H_1\}( {\tPhi}^s_\eps({\tz}^\st_\eps))\ -
\{P,H_1\}( {\tPhi}^s_\eps {\Omega}_\eps^\st({\tz}^\st_\eps) )\big] \omega(t_0+s) \, d s.
\end{split}
\end{equation}

We split the  integral on the right-hand side of the above into two:
\begin{equation}\label{FTC_2}
\begin{split}
I=&-\eps\int_0^{k\log(1/\eps)} \big( \{P,H_1\}( {\tPhi}^s_\eps({\tz}^\st_\eps)) -
\{P,H_1\}( {\tPhi}^s_\eps {\Omega}_\eps^\st({\tz}^\st_\eps) ) \big) \omega(t_0+s) \, d s,
\\ II=&-\eps
\int_{k\log(1/\eps)}^{\infty} \big(\{P,H_1\}( {\tPhi}^s_\eps({\tz}^\st_\eps))  -
\{P,H_1\}( {\tPhi}^s_\eps {\Omega}_\eps^\st({\tz}^\st_\eps) ) \big)\omega(t_0+s) \, d s,
\end{split}
\end{equation}
for some $k>0$.

For the second integral, since \[\big[\{P,H_1\}( {\tPhi}^s_\eps({\tz}^\st_\eps)) -
\{P,H_1\}( {\tPhi}^s_\eps {\Omega}_\eps^\st({\tz}^\st_\eps) )\big]\omega(t_0+s)\to 0\] exponentially fast, for any $k>0$ we have that  $II=O_{C^1}(\eps^{\varrho_2})$ for some $\varrho_2>0$.

For the first integral, by applying the Gronwall Inequality -- III \ref{lem:gronwall-III} and choosing $k>0$ sufficiently small, we can replace the terms depending on the perturbed flow by corresponding terms depending on the unperturbed flow, while making an error of order $O_{C^1}(\eps^{\varrho_1})$, for some $\varrho_1\in(0,1)$, obtaining
\begin{equation*}
\begin{split}
I=&-\eps\int_0^{k\log(1/\eps)} \big[\{P,H_1\}( {\tPhi}^s_0({\tz}^\st_0)) -
\{P,H_1\}( {\tPhi}^s_0{\Omega}_0^\st({\tz}^\st_0 ))\big] \omega(t_0+s) \, d s\\&+O_{C^1}(\eps^{1+\varrho_1}).
\end{split}
\end{equation*}

Since also \[\big[\{P,H_1\}( {\tPhi}^s_0({\tz}^\st_0)) -
\{P,H_1\}]( {\tPhi}^s_0 {\Omega}_0^\st({\tz}^\st_0) )\big]\omega(t_0+s)\to 0\] exponentially fast, we
can replace the above integral from $0$ to ${k\log(1/\eps)}$ by the improper integral from  $0$ to $+\infty$, by making an error of order $O(\eps^{1+\rho_2})$.

Thus,
\begin{equation*}
\begin{split}
I=&-\eps\int_0^{+\infty} \big[ \{P,H_1\}( {\tPhi}^s_0({\tz}^\st_0)) -
\{P,H_1\}( {\tPhi}^s_0{\Omega}_0^\st({\tz}^\st_0 ) \big] \omega(t_0+s) \, d s\\&+O_{C^1}(\eps^{1+\varrho}),
\end{split}
\end{equation*}
where $\varrho=\min\{\varrho_1,\varrho_2\}$.

Combining $I$ and $II$ we obtain
\begin{equation}\label{FTC_3}
\begin{split}
P(\tz^\st_\eps) - P({\Omega}^\st_\eps(\tz^\st_\eps))
= & -\eps\int_0^{+\infty} \big[ \{P,H_1\}( {\tPhi}^s_0(\tz^\st_0)) -
\{P,H_1\} ]( \tPhi^s_0 {\Omega}_0^\st(\tz^\st_0) ) \big] \omega(t_0+s) \, d s\\
&+O_{C^1}(\eps^{1+\rho}).
\end{split}
\end{equation}

Using that $\tz_0=(I,\phi,p_0(\tau),q_0(\tau),t_0)$ and $\tz_0^\pm =(I,\phi,0,0,t_0)$, we have
\begin{equation}\label{FTC_4}
\begin{split}
P({\tz}^\st_\eps) - P({\Omega}^\st_\eps({\tz}^\st_\eps))
= & -\eps\int_0^{+\infty}\{P,H_1\}  (I,\phi+\nu(I) s, p(\tau+s), q(\tau+s)) \omega(t_0+s)\, d s\\&+O_{C^1}(\eps^{1+\rho}).
\end{split}
\end{equation}

A similar computation for a point ${\tz}^\un_\eps$ in $W^\un(\Lambda_\eps(\theta^{t_0}\omega))$,
yields
\begin{equation}\label{FTC_5}
\begin{split}
P({\tz}^\st_\eps) - P({\Omega}^\un_\eps({\tz}^\un_\eps))
= & -\eps\int_{-\infty}^{+\infty}  \{P,H_1\}  (I,\phi+\nu(I) s, p(\tau+s), q(\tau+s))   \omega(t_0+s) \, d s \\
&+O_{C^1}(\eps^{1+\rho}),
\end{split}
\end{equation}
where $\Omega^\un_\eps({\tz}^\un_\eps)$ be the foot-point of the unique unstable fiber of   $W^\un_\eps(\Lambda_\eps(\theta^{t_0}\omega))$ through ${\tz}^\un_\eps$.

Subtracting \eqref{FTC_4} from \eqref{FTC_5} yields
\begin{equation}\label{FTC_6}
\begin{split}
P({\tz}^\st_\eps) - P({\tz}^\un_\eps)=& P({\Omega}^\st_\eps({\tz}^\un_\eps))-P({\Omega}^\un_\eps({\tz}^\st_\eps))\\
&-\eps\int_{-\infty}^{+\infty} \{P,H_1\}  (I,\phi+\nu(I) s, p(\tau+s), q(\tau+s)) \omega(t_0+s) \, d s\\
&+O(\eps^{1+\rho}).
\end{split}
\end{equation}

The points  ${\Omega}^\un_\eps({\tz}^\un_\eps)$ and ${\Omega}^\st_\eps({\tz}^\st_\eps)$ are   in
$\Lambda_\eps(\theta^{t_0}\omega)$.
Since $P$ has a critical point at $(0,0)$ and $P({\Omega}^{\st}_0({\tz}_0))=P({\Omega}^{\un}_0({\tz}_0))$
it follows that
\[\|P({\Omega}^{\st,\un}_\eps( {\tz}^{\st,\un}_\eps))
- P({\Omega}^{\st,\un}_0({\tz}_0))\|_{C^{1}} \le C\eps^2,\]
and therefore
\[\|P({\Omega}^\st_\eps({\tz}^\st_\eps)) -P({\Omega}^\un_\eps({\tz}^\un_\eps))\|_{C^{1}}=O(\eps^2).\]

Therefore, \eqref{FTC_6}  immediately implies \eqref{eqn:Pun_minus_Pst}, which concludes the proof.
\end{proof}

\subsection{Time invariance}

Consider the \emph{Melnikov function} that appears on the right-hand side of Proposition \ref{thm:transverse_homoclinic}.
\begin{equation}\label{eqn:tau_integral}
\begin{split}
(I,\phi,\tau,t_0)\mapsto & M^P( I,\phi, \tau, t_0 )\\=& \int_{-\infty}^{\infty}  \{P,H_1\}(I,\phi+\nu(I) s,p_0(\tau+ s ),q_0(\tau+ s)) \omega(t_0+s)  \, d s
\end{split}
\end{equation}

The following is immediate:

\begin{lem}\label{lem:zero_trans} Let $t_0\in Q_{A_\delta,T_\delta}$.
If the mapping \[\tau\mapsto M^P(I,\phi,\tau,t_0)\] has a non-degenerate zero $\tau^*=\tau^*(I,\phi,t_0)$, then  there exists $0<\eps_1<\eps_0$ such that
\[\tau \mapsto P(\tz^\st_\eps) -P(\tz^\un_\eps)\]
has  a non-degenerate zero $\tau^*_\eps=\tau^*(I,\phi,t_0)+O_{C^1}(\eps^{1+\varrho})$ for all $0<\eps<\eps_1$, which in turn implies the existence of a transverse homoclinic intersection of $W^\st(\Lambda_\eps(\theta^{t_0}\omega))$ and $W^\un(\Lambda_\eps(\theta^{t_0}\omega))$ for all $0<\eps<\eps_1$.
\end{lem}

We will show that $W^\st(\Lambda_\eps(\theta^{t_0}\omega))$ and $W^\un(\Lambda_\eps(\theta^{t_0}\omega))$ intersect transversally in Section \ref{sec:existence_transverse}.

Towards this goal, we start by making  the following key observation:

\begin{rem}\label{rem:important}
The argument in the proof of Proposition \ref{thm:transverse_homoclinic} shows that the  mapping
\[(I,\phi,\tau,\t)\mapsto M^P(I,\phi,\tau,\t)\] is well-defined for all $\t\in\mathbb{R}$, and
not only for $\t\in Q_{A_\delta,T_\delta}(\omega)$.
Indeed, for any $\t\in\mathbb{R}$, by Lemma \ref{lem:eta}  we have
$|\omega(\t+s)| \le A_{\theta^{\t}\omega}+B|s|$ for all $s$.
Therefore, for  $\omega$ and $\t$ fixed, the integrand in \eqref{eqn:tau_integral}  is exponentially convergent to $0$ as $s\mapsto \pm\infty$.
\end{rem}

\begin{prop}[Time invariance of $M^P$]\label{prop:time_invariance}{ $ $}

\begin{itemize}
\item [(i)]
For any $\varsigma\in\mathbb{R}$ we have
\begin{equation}\label{eqn:M_invariance}
   M^P(I,\phi,\tau,  \t)=M^P(I,\phi+\nu(I)\varsigma,\tau+\varsigma,  \t+\varsigma ).
\end{equation}

\item [(ii)]If the mapping \[\tau\mapsto M^P(I,\phi,\tau, \t ),\]
has a non-degenerate zero at $\tau^*(I,\phi,\t)$,
then for any $\varsigma\in\mathbb{R}$, the mapping
\[\tau\mapsto M^P(I,\phi+\nu(I)\varsigma,\tau, \t+\varsigma )\] has a non-degenerate zero $\tau^*(I,\phi+\nu(I)\varsigma, \t+\varsigma )$, and
\begin{equation}\label{eqn:tau_invariance}
   \tau^*(I,\phi+\nu(I)\varsigma, \t+\varsigma )=\tau^*(I,\phi,\t )+\varsigma.
\end{equation}

\item [(iii)] If the mapping \[ \t  \mapsto M^P(I,\phi,\tau, \t ),\]
has a non-degenerate zero at $t^*(I,\phi,\tau)$,
then for any $\varsigma\in\mathbb{R}$, the mapping
\[\t\mapsto M^P(I,\phi+\nu(I)\varsigma,\tau+\varsigma, \t )\] has a non-degenerate zero $\t^*(I,\phi+\nu(I)\varsigma,\tau+\varsigma)$, and
\begin{equation}\label{eqn:t_invariance}
   \t^*(I,\phi+\nu(I)\varsigma,\tau+\varsigma)=t^*(I,\phi,\tau)+\varsigma.
\end{equation}
\end{itemize}
\end{prop}

\begin{proof}
In \eqref{eqn:tau_integral} make the change of variable $t\mapsto \varsigma+t$, obtaining
\begin{equation}\label{eqn:change_var}
\begin{split}
\int_{-\infty}^{\infty}  &\{P,H_1\}(I,\phi+\nu(I) s,p_0(\tau+s),q_0(\tau+s)) \omega(\t+s)  \, d s\\
=& \int_{-\infty}^{\infty}  \{P,H_1\}(I,\phi+\nu(I)\varsigma+\nu(I) s,p_0(\tau+\varsigma+s),q_0(\tau+\varsigma+s)) \omega(\t+\varsigma+s )  \, d s.
\end{split}
\end{equation}

Thus, if $\tau^*(I,\phi,\t)$ is a non-degenerate zero of the first integral, then $\tau^*(I,\phi,\t)+\varsigma$ is a non-degenerate zero for the second integral,
and hence $\tau^*(I,\phi,\t)+\varsigma=\tau^*(I,\phi+\nu(I)\varsigma,\t+\varsigma)$.

Similarly, if $\t^*(I,\phi,\tau)$ is a  non-degenerate zero for the first integral, then $\t^*+\varsigma= \t^*(I,\phi,\tau)+\varsigma$  is non-degenerate  zero for the second integral.
\end{proof}

\subsection{Existence of transverse homoclinic intersections}
\label{sec:existence_transverse}

Let $\mathscr{F}:\mathbb{R}\to \mathbb{R}$ a $C^1$-function with the property that $\mathscr{F}(\t)\to 0$ as $\t\to+\infty$ and   $D{\mathscr{F}}(\t)\to 0$ as $\t\to+\infty$ exponentially fast, where $D({\cdot})$ denotes the derivative of a function.

\begin{lem}\label{lem:process}
Let
\[M(\t)=\int_{-\infty}^{\infty} \mathscr{F}(s)\theta^{\t}\omega(s) ds. \]

Then the process $M(\t)$ is a stationary Gaussian process with expectation
\begin{equation}\label{eqn:Mexp}
 E[M(\t)] =0
\end{equation}
and autocorrelation function
\begin{equation}\label{eqn:Mauto}
\begin{split}
\rho(h)=&E[M(\t)  M(\t+h)]\\=&\int_{-\infty}^{\infty} \int_{-\infty}^{\infty}\mathscr{F}(s_1) \mathscr{F}(s_2)r(s_2-s_1+h)ds_1ds_2,\end{split}\end{equation}
where $r$ is the autocorrelation function of $\eta$.
The integral is independent of $\t$ as the process is stationary.
\end{lem}
\begin{proof}
The proof is similar to  \cite[Lemma 4.4]{yagasaki2018melnikov}.
Since $\eta(\t)$ is a stationary Gaussian process with mean $0$, $M(\t)$ is also a stationary Gaussian process with mean $0$.

The variance (zeroth spectral moment) of $M(\t)$ is given by  \[\chi_0=\rho(0)=\int_{-\infty}^{\infty}\int_{-\infty}^{\infty} \mathscr{F}(s_1)\mathscr{F}(s_2)r(s_2-s_1) ds_1ds_2.\]
Performing a change of variable $(s_1,s_2-s_1)\mapsto (\t,s)$ we can write
\[\chi_0=\rho(0)=\int_{-\infty}^{\infty}\int_{-\infty}^{\infty} \mathscr{F}(\t)\mathscr{F}(\t+s)r(s) dtds.\]

By definition, the autocorrelation function of $M(\t)$  is given by
\[\rho(h):=E\left[M(\t+h)M(\t)\right].\]

Taking the second derivative with respect to $\t$ of
\[\rho(h)=\int_{-\infty}^{\infty}\int_{-\infty}^{\infty} \mathscr{F}(s_1)\mathscr{F}(s_2)r(s_2-s_1+h) ds_1 ds_2\]
and applying integration by parts twice we obtain
\[\frac{d^2}{dh^2}\rho(h)=-\int_{-\infty}^{\infty}\int_{-\infty}^{\infty}D{\mathscr{F}}(s_1)D{\mathscr{F}}(s_2)r(s_2-s_1+h) ds_1ds_2,\]
where  $D{\mathscr{F}}(s_1)=\frac{\partial \mathscr{F}}{\partial s_1 }$ and $D{\mathscr{F}}(s_2)=\frac{\partial \mathscr{F}}{\partial s_2}$.
For  integration by parts  we have used that $\lim_{s\to \pm\infty}\mathscr{F}(s)= 0$ exponentially fast together with its derivative, and that $r$ is a bounded function together with its derivative.

Performing a change of variable $(s_1,s_2-s_1)\mapsto (\t,s)$ and setting $h=0$ we obtain that the  second spectral moment of $M$ is given by
\[\chi_2=-\frac{d^2}{dh^2}\rho(h)_{\mid h=0}=\int_{-\infty}^{\infty}\int_{-\infty}^{\infty} D{\mathscr{F}}(\t)D{\mathscr{F}}(\t+s)r(s) dt ds.\]
\end{proof}

Denote \begin{equation}\begin{split}
\mathscr{P}(I,\phi,\tau,s)=&\{P,H_1\}\left(I,\phi+\nu(I)s,p_0(\tau+s),q_0(\tau+s)\right).
\end{split}\end{equation}
When $( I,\phi,\tau)$ are fixed and we only want to emphasize the dependence on $s$ we denote the above function by $\mathscr{P}(s)$.
Note that $\mathscr{P}(s)$ converges to $0$ exponentially fast together with its derivative as $s\to\pm\infty$.

Define  the \emph{Melnikov stochastic process}
\begin{equation}
\label{eqn:melnikov_process}
\t\in\mathbb{R}\mapsto M^P(\t):= \int_{-\infty} ^{+\infty} \mathscr{P}(s)\theta^{\t}\omega(s) d s.
\end{equation}

\begin{prop}\label{prop:Mprocess}
The Melnikov process \eqref{eqn:melnikov_process}  is a stationary Gaussian process with expectation
\begin{equation}\label{eqn:Mexp}
 E[M^P(\t)] =0
\end{equation}
and autocorrelation function
\begin{equation}\label{eqn:Mauto}
\begin{split}
\rho(h)=&E[M^P(\t)  M^P(\t+h)]\\=&\int_{-\infty}^{\infty} \int_{-\infty}^{\infty}\mathscr{P}(s_1) \mathscr{P}(s_2)r(s_2-s_1+h)ds_1ds_2,\end{split}\end{equation}
with the right-side independent of $\t$ as the process is stationary.
The zeroth  spectral moments is:
\[\chi^P_0 =\rho(0)=\int_{-\infty}^{\infty} \int_{-\infty}^{\infty}\mathscr{P}(s_1) \mathscr{P}(s_1+s)r(s)ds_1 ds\]
and the second spectral moment is:
\[\chi^P_2=-\left(\frac{d^2}{d h^2}\rho (h)\right )_{\mid h=0}=\int_{-\infty}^{\infty}\int_{-\infty}^{\infty} D{\mathscr{P}}(s_1) D{\mathscr{P}}(s_1+s)r(s) ds_1 ds.\]
\end{prop}
\begin{proof} It follows immediately from Lemma \ref{lem:process}.
\end{proof}

We will make the following assumption on the spectral moments of $M^P$:

\begin{equation}\label{eqn:spectral_moments}\tag{SMP}
\chi^P_0>0 \textrm{ and }\chi^P_2>0.
\end{equation}

\begin{prop} \label{prop:nondegenerate}
Assume   condition \eqref{eqn:spectral_moments}.
Fix $( I^*,\phi^*,\tau^*)$. Then  the mapping
\[\t\in\mathbb{R} \mapsto  M^P( I^*,\phi^*, \tau^*,\t)\]
has a non-degenerate zero $\t^*\in\mathbb{R}$.
\end{prop}
\begin{proof}
Applying the Rice's Formula \cite{rice1939distribution,lindgren2012stationary} to the Melnikov process \[\t\mapsto M^P(\t),\] the number $N_T $ of zeros of $M^P(\t)$ on the interval $[0,T]$ has expectation
\begin{equation}\label{eqn:rice}
  E[N_T]=\frac{T}{\pi}\sqrt{\frac{\chi^P_2}{\chi^P_0}}e^{-\frac{ E[M^P(\t)]^2}{2\chi^P_0}}=\frac{T}{\pi}\sqrt{\frac{\chi^P_2}{\chi^P_0}},
  \end{equation}
and $M^P(\t)$ has almost surely no tangential zeroes. This yields the desired conclusion.
\end{proof}
%

\begin{proof}[\bf Proof of Theorem \ref{prop:transverse}]
Let $\t^*=\t^*(I^*,\phi^*,\tau^*)$ be a non-degenerate zero from Proposition \ref{prop:nondegenerate}, for some $I^*,\phi^*,\tau^*$ fixed.
Then  $\tau^*$ is a zero of \[\tau \mapsto M^P(I^*,\phi^*,\tau,\t^*).\]
Moreover,   from  Proposition \ref{prop:time_invariance}, \eqref{eqn:tau_invariance} and \eqref{eqn:t_invariance}, it follows that $\tau^*$ is a non-degenerate zero.
Indeed, the fact that $\t^*$ is a non-degenerate zero implies that there are $\varsigma>0$ arbitrarily close to $0$ such that one of the following holds:
\[\begin{split}
M(I,\phi+\nu(I)(-\varsigma),   \tau^*-\varsigma ,t^*_0-\varsigma)<0 < M(I,\phi+\nu(I) \varsigma , \tau^*+\varsigma,t^*_0+\varsigma),\\
M(I,\phi+\nu(I)(-\varsigma),   \tau^*-\varsigma ,t^*_0-\varsigma)>0  > M(I,\phi+\nu(I) \varsigma , \tau^*+\varsigma,t^*_0+\varsigma).
\end{split}
\]
This implies that $\tau^*$ is a non-degenerate zero of  \[\tau \mapsto M^P(I^*,\phi^*,\tau,\t^*).\]
Therefore, $\tau^*$ is (locally) uniquely defined by $I^*,\phi^*,\t^*$, so we can write $\tau^*=\tau^*(I^*,\phi^*,\t^*)$.

If $t_0:=\t^*\in  Q_{A_\delta,T_\delta}(\omega)$, then
$W^\un(\Lambda_\eps(\theta^{t_0}\omega))$ and $W^\st(\Lambda_\eps(\theta^{t_0}\omega))$ are well defined. By Lemma \ref{lem:zero_trans}, it follows that  they intersect transversally for $\eps<\eps_1$.

If $\t^*\not\in  Q_{A_\delta,T_\delta}(\omega)$  then take  $\varsigma^*\in\mathbb{R}$ such that $t_0:=\t^*+\varsigma^*\in Q_{A_\delta,T_\delta}(\omega)$. Such a $\varsigma^*$ always exists due to Proposition \ref{lem:Yagasaki}. By Proposition \ref{prop:time_invariance} we have that
$\tau^*(I^*,\phi^*+\nu(I^*)\varsigma,t^*_0+\varsigma^*)=\tau^* +\varsigma^*$ is a non-degenerate zero of
\[\tau \mapsto M^P(I^*,\phi^*+\nu(I^*)\varsigma,\tau^*,\t^*+\varsigma^*).\]
This implies that $W^\un(\Lambda_\eps(\theta^{\t^*+\varsigma^*}\omega))$ and $W^\st(\Lambda_\eps(\theta^{\t^*+\varsigma^*}\omega))$ intersect transversally for all $0<\eps<\eps_1$ for some $\eps_1<\eps_0$ sufficiently small.
\end{proof}

\begin{rem}
Recall that  $\tau$ corresponds to a `position' along the homoclinic orbit of the unperturbed system.
The typical way to use Melnikov theory to show existence of transverse intersection of the perturbed stable and unstable manifolds is to vary the position $\tau$ until we reach a place where the distance between the manifolds is $0$.

We emphasize that in our argument above -- Proposition \ref{prop:transverse} -- we first fix the position $\tau$. Then we show that there is a time $t$ when the distance between the manifolds is $0$ at that  fixed position. In other words, we  wait  until the noise  pushes  the manifolds to cross one another.

An extra complication is that the perturbed stable and unstable manifolds are not well-defined for all times. We have to adjust $t$ to a time where the perturbed stable and unstable manifolds are well-defined. To achieve this, we also adjust the location $\tau$.
For this, we use the time invariance of the Melnikov integral -- Proposition \ref{prop:time_invariance}.
\end{rem}

Theorem \eqref{prop:transverse}  gives us a homoclinic point $\tz $ at the transverse intersection between
$W^{\st}  (\Lambda_\eps(\theta^{t_0}(\omega))$ and $W^{\un}  (\Lambda_\eps(\theta^{t_0}(\omega))$, for some $t_0\in Q_{A_\delta,T_\delta}(\omega)$.
Consider the homoclinic orbit $\tPhi^t_\eps(\tz_\eps)$.
Unlike  the deterministic case, we cannot guarantee that $\tPhi^t_\eps(\tz )\in W^{\st}  (\Lambda_\eps(\theta^{t_0+t}(\omega))\cap W^{\un}  (\Lambda_\eps(\theta^{t_0+t}(\omega))$, since $t_0+t$ may not be in $Q_{A_\delta,T_\delta}(\omega)$, and so the stable and unstable manifolds corresponding
to $\theta^{t_0+t}(\omega)$ are not guaranteed to exist (as graphs).
Nevertheless, the  homoclinic orbit $\tPhi^t_\eps(\tz)$ is asymptotic to the normally hyperbolic invariant manifold in both forward and backwards times, as in the deterministic case. This is given by the following:

\begin{cor}\label{cor:convergence}
Let $t_0\in Q_{A_\delta,T_\delta}(\omega)$.
If $\tz\in W^{\st}  (\Lambda_\eps(\theta^{t_0}(\omega))$ then there exists a unique point $\tz^+\in \Lambda_\eps(\theta^{t_0}(\omega))$ such that $d(\tPhi^t_\eps(\tz),\tPhi^t_\eps(\tz^+))\to 0$ as $t\to +\infty$.

Assuming condition \eqref{eqn:H_1}, it follows that $\tPhi^t_\eps(\tz)$ approaches $\tLambda_0$ as $t\to +\infty$.

Similarly,  if $\tz\in W^{\un}  (\Lambda_\eps(\theta^{t_0}(\omega))$ then there exists a unique point $\tz^-\in \Lambda_\eps(\theta^{t_0}(\omega))$ such that $d(\tPhi^t_\eps(\tz),\tPhi^t_\eps(\tz^-))\to 0$ as $t\to -\infty$. Assuming condition \eqref{eqn:H_1}, it follows that $\tPhi^t_\eps(\tz)$ approaches $\Lambda_0$ as $t\to -\infty$.
\end{cor}

\begin{proof}
If $\tz\in W^{\st}  (\Lambda_\eps(\theta^{t_0}(\omega))$  then $\tPhi^t_\eps(\tz)$ is on the stable manifold of  $\Lambda_\eps(\theta^{t_0+t}(\omega))$ for the modified system \eqref{eqn:z_hat}.
This implies that $d(\tPhi^t_\eps(\tz),\tPhi^t_\eps(\tz^+))\to 0$ as $t\to+\infty$.
Since for our system $\Lambda_\eps(\theta^{t_0}(\omega))=\Lambda_0$ for all $t$ and $\omega$, it follows that $\tPhi^t_\eps(\tz)$ approaches $\Lambda_0$ as $t\to +\infty$.

A similar argument holds for the unstable manifold.
\end{proof}

%
%

\section{Existence of orbits that increase in action}

\subsection{Random scattering map}
\label{sec:random_scattering}
In this section we adapt the theory of the scattering map developed in \cite{DelshamsLS06c} for the case of random perturbations.
Our construction is very similar with the time-dependent scattering theory for general vector fields  developed in \cite{BlazevskiL11}.
The most significant difference is that in our case the scattering map depends on the realization of the stochastic process in a measurable fashion.

Let $t_0\in Q_{A_\delta,T_\delta}(\omega)$.

For a point $\tz\in W^\st(\Lambda_\eps(\theta^{t_0}\omega))$ (resp.\ $\tz\in W_\Lambda^\un(\Lambda_\eps(\theta^{t_0}\omega))$), we
denote by $\tz^+$ (resp.\ $\tz^-$) the unique point in $\Lambda_\eps(\theta^{t_0}\omega)$ which satisfies $\tz\in W^\sst(\tz^+,\theta^{t_0}\omega)$ (resp.  $\tz\in   W^\uun(\tz^-,\theta^{t_0}\omega)$). The stable and unstable fibers referred above are given by Theorem \ref{thm:Bates1}.

Then the canonical projections
\begin{equation}  \label{waveoperators}
\begin{split}
\Omega^{+}_\eps(\cdot,\theta^{t_0}\omega) : W^\st(\Lambda_\eps(\theta^{t_0}\omega))&\to \Lambda_\eps(\theta^{t_0}\omega), \\
\Omega^{+}_\eps(\tz, \theta^{t_0}\omega )=& \tz^{+},\\
\Omega^{-}_\eps (\cdot, \theta^{t_0}\omega ): W^\un(\Lambda_\eps(\theta^{t_0}\omega)) & \to \Lambda_\eps(\theta^{t_0}\omega), \\
\Omega^{-}_\eps(\tz,\theta^{t_0}\omega )=& \tz^{-}.
\end{split}
\end{equation}
are well defined $\C^{\ell-1}$ maps in $\tz$ and measurable in $\omega$.

Now, assume  there is  a homoclinic manifold
$\Gamma_\eps(\theta^{t_0}\omega) \subset W^\sst(\Lambda_\eps(\theta^{t_0}\omega))\cap W^\uun(\Lambda_\eps(\theta^{t_0}\omega))$
satisfying the following  conditions:
\begin{equation}\label{intersection}
\begin{split}
&T_{\tz} M = T_{\tz} W^\st(\Lambda_\eps(\theta^{t_0}\omega))+ T_{\tz} W^\un(\Lambda_\eps(\theta^{t_0}\omega)),\\
&T_{\tz} W^\st(\Lambda_\eps(\theta^{t_0}\omega)) \cap T_{\tz} W^\un(\Lambda_\eps(\theta^{t_0}\omega)) = T_{\tz} \Gamma_\eps(\theta^{t_0}\omega),\\
&T_{\tz} \Gamma_\eps(\theta^{t_0}\omega)  \oplus T_{\tz}  W^\sst(\tz^+,\theta^{t_0}\omega) = T_{\tz} W^\st(\Lambda_\eps(\theta^{t_0}\omega)),\\
&T_{\tz} \Gamma_\eps(\theta^{t_0}\omega)  \oplus T_{\tz}  W^\uun(\tz^-,\theta^{t_0}\omega) = T_{\tz} W^\un(\Lambda_\eps(\theta^{t_0}\omega)),
\end{split}
\end{equation}
for all $\tz\in  \Gamma_\eps(\theta^{t_0}\omega)$.

The first two conditions in \eqref{intersection} say that $W^\st(\Lambda_\eps(\theta^{t_0}\omega))$ and $W^\un(\Lambda_\eps(\theta^{t_0}\omega))$ intersect transversally
along $\Gamma_\eps(\theta^{t_0}\omega)$,
and the last two conditions say that $\Gamma_\eps(\theta^{t_0}\omega)$ is transverse to the stable and unstable foliations.
We now  consider the canonical projections $\Omega_\eps ^{\pm}(\cdot,\theta^{t_0}\omega)$
\eqref{waveoperators} restricted to $\Gamma_\eps(\theta^{t_0}\omega) $.
Under the assumption \eqref{intersection} we have that $\Gamma_\eps(\theta^{t_0}\omega) $ is $\C^{\ell-1}$ and that $\Omega_\eps^\pm(\cdot,\theta^{t_0}\omega)$ are $\C^{\ell-1}$ local diffeomorphisms  from $\Gamma_\eps(\theta^{t_0}\omega)$ to $\Lambda_\eps(\theta^{t_0}\omega)$.

Let us further assume further that $\Gamma_\eps(\theta^{t_0}\omega)$ is a \emph{homoclinic channel}, that is,
\[\Omega^\pm_\eps (\cdot,\theta^{t_0}\omega):\Gamma_\eps(\theta^{t_0}\omega)
\to U_{\eps}^\pm (\theta^{t_0}\omega):= \Omega_\eps ^{\pm} (\Gamma_\eps(\theta^{t_0}\omega),\theta^{t_0}\omega)
\] is
a  $C^{\ell-1}$-diffeomorphism.

\begin{defn}\label{def:scattering}
The \emph{random scattering map} associated to $\Gamma_\eps(\theta^{t_0}\omega)$ is defined as
\[
\sigma_\eps(\cdot,\theta^{t_0}\omega): U_{\eps}^-(\theta^{t_0}\omega) \to U^+_{\eps}(\theta^{t_0}\omega),\]
given by
\begin{equation} \label{eq:scattering}
\sigma_\eps(\cdot,\theta^{t_0}\omega)=\Omega^+_\eps(\cdot,\theta^{t_0}\omega)  \circ \left( \Omega^-_\eps(\cdot,\theta^{t_0}\omega)\right )^{-1}.
\end{equation}
\end{defn}

For each fixed path $\omega$, the scattering map is $\C^{\ell-1}$, and it depends on $\omega$ in a measurable fashion.

When there is no dependence on a random variable, Definition \ref{def:scattering} translates into the definition of the scattering map in the deterministic case
\cite{DelshamsLS06c}.

\subsection{The scattering map for the unperturbed  pendulum-rotator system}
\label{sec:scattering_rot_pend}
For the unperturbed system the definition of the scattering map from above translates into the standard definition of the scattering map as  in \cite{DelshamsLS06c}.

Since we have $W^\st(\Lambda_0)=W^\un(\Lambda_0)$ and for each $z\in\Lambda_0$, $W^\st(z)=W^\un(z)$, the corresponding scattering map $\sigma_0$ is defined on the whole $\Lambda_0$ as the identity map. Thus, $\sigma_0(z^-)=z^+$ implies $z^-=z^+$.
Expressed in terms of the action-angle coordinates $(I,\phi)$ of the rotator, we have  \begin{equation}\label{eqn:sigma0i}\sigma_0(I,\phi)=(I,\phi).\end{equation}

In the next section, we provide a formula to estimate the effect of the scattering map on the action of the rotator.

\subsection{Change in action by the scattering map}
\label{sec:change_in_action}
%

Denote
\begin{equation}\begin{split}
\mathscr{I}(I,\phi,\tau, s)= &\{I,H_1\}\left(I,\phi+\nu(I) s,p_0(\tau+ s ),q_0(\tau+ s)\right)
\\
  & -\{I,H_1\}\left(I,\phi+\nu(I) s,0,0\right)
\end{split}\end{equation}
and, for $(I,\phi,\tau)$  fixed, let
\begin{equation}\label{eqn:tau_I_integral}
 \t \mapsto M^I(\t)= \int_{-\infty}^{\infty} \mathscr{I}(I,\phi,\tau,s) \theta^{\t}\omega(s)  \, d s .
\end{equation}

\begin{prop}
The stochastic process \begin{equation}\begin{split}
\t\mapsto M^I(\t)\end{split}\end{equation}
is a stationary Gaussian process with mean
\[E[M^I(\t)]=0,\]
and autocorrelation \[\rho^I(\sigma):=E\left[M^I(\tau+\sigma)M^I(\tau)\right] \] given by
\[\rho^I(h)=\int_{-\infty}^{\infty}\int_{-\infty}^{\infty} \mathscr{I}(s_1)\mathscr{I}(s_2)r(s_2-s_1+h) ds_1ds_2.\]

Hence $M^I(\t)$ is ergodic.
\end{prop}
\begin{proof}
It follows immediately by applying Lemma \ref{lem:process} to $s\mapsto \mathscr{I}(I,\phi,\tau,s)$.
\end{proof}

Define
\begin{equation}\label{}
\chi^I_0=\rho^I(0)\textrm{ and }\chi^I_2=-\frac{d^2}{dh^2}_{\mid h=0}\rho^I(h).
\end{equation}

We will make the following assumption on the spectral moments of  $M^I$:

\begin{equation}\label{eqn:spectral_moments_I}\tag{SMI}
\chi^I_0>0 \textrm{ and }\chi^I_2>0.
\end{equation}

\begin{prop}\label{prop:Rice2}
Assume the condition \eqref{eqn:spectral_moments_I}.  Let $v>0$.

Then the number of times $N_T$ the process $M^I(\t)$ crosses the value  $v$ from above as well as from below
on the interval $[0,T]$ has expectation \begin{equation}
E[N_T]=\frac{T}{\pi}\sqrt{\frac{\chi^I_2}{\chi^I_0}}\exp\left(-\frac{(v-E[M^I(\t)])^2}{2\chi^I_0}\right)=
\frac{T}{\pi}\sqrt{\frac{\chi^I_2}{\chi^I_0}}\exp\left(-\frac{ v ^2}{2\chi^I_0}\right).
\end{equation}
\end{prop}

\begin{proof}
Apply Rice's formula (see \cite{rice1939distribution,lindgren2012stationary}).

\end{proof}

\begin{proof}[\bf Proof of Theorem \ref{prop:change_in_I}]

Suppose $t_0\in Q_{A_\delta,T_\delta}$ and $\tz_\eps\in W^\un(\Lambda_\eps(\theta^{t_0}\omega))\cap W^\st(\Lambda_\eps(\theta^{t_0}\omega))$ is a transverse homoclinic point for $0<\eps<\eps_1$, as given by Theorem \ref{prop:transverse}.
Then there exists a homoclinic channel $\Gamma_\eps (\theta^{t_0}\omega)\subset W^\st(\Lambda_\eps(\theta^{t_0}\omega)$ containing $\tz_\eps$.

We recall that for Theorem \ref{prop:change_in_I} we assume \eqref{eqn:H_1}, which means that $\Lambda_\eps(\theta^{t_0}\omega)=\Lambda_0$ for $t_0\in Q_{A_\delta,T_\delta}$. This implies that the perturbed inner dynamics restricted to the RNHIM coincides the inner dynamics in the unperturbed case, and, in particular that it preserved  the action coordinate $I$ along orbits.

A computation similar to that in the proof of Proposition \ref{thm:transverse_homoclinic} yields
\begin{equation}\label{eqn:change_in_I_non_ham}\begin{split}
  I\left(\tz^{+}_{\eps}\right)-I\left(\tz^-_{\eps}\right)
= &\eps\int_{-\infty}^{+\infty} \big[\{I,H_1\}\left(I,\phi+\nu(I) s,p_0(\tau^*+ s ),q_0(\tau^*+ s)\right)
\\
  &\qquad\qquad -\{I,H_1\}\left(I,\phi+\nu(I) s,0,0\right)\big]\omega(t_0+s) \,  d s
  \\
  &+O(\eps^{1+\rho})\\
  = &\eps M^I(I,\phi,\tau,\t)+O(\eps^{1+\rho}),
\end{split}\end{equation}
for $0<\varrho<1$; for details, see, e.g., \cite{gidea2021global}.

Following the same computation as in the proof of Proposition \ref{prop:time_invariance}
we obtain the following time invariance relation
\begin{equation}\label{eqn:M_invariance}
   M^I(I,\phi,\tau,  \t )=M^I(I,\phi+\nu(I)\varsigma,\tau+\varsigma, \t+\varsigma ).
\end{equation}
for all $\varsigma\in\mathbb{R}$.

If we apply the flow $\tPhi^\varsigma_\eps$ to the point $\tz_\eps$, while $\tPhi^\varsigma_\eps(\tz_\eps)$ may no longer stay in $W^\un(\Lambda_\eps(\theta^{t_0}\omega)\cap W^\st(\Lambda_\eps(\theta^{t_0}\omega)$, it remains asymptotic in both forward and backward time to $\Lambda_0$.

More precisely, we have
\begin{equation}\label{eqn:biasimptotic}
\begin{split}
d(\tPhi^{t}_\eps(\tPhi^\varsigma_\eps(\tz_\eps),\tPhi^{t}_\eps(\tPhi^\varsigma_\eps(\tz^+_\eps))\to 0 \textrm { when } t\to+\infty,\\
d(\tPhi^{t}_\eps(\tPhi^\varsigma_\eps(\tz_\eps),\tPhi^{t}_\eps(\tPhi^\varsigma_\eps(\tz^-_\eps))\to 0 \textrm { when } t\to-\infty.
\end{split}
\end{equation}

We have
\begin{equation}\label{eqn:change_in_I_non_ham_2}\begin{split}
  I&\left(\tPhi^\varsigma_\eps(\tz^{+}_{\eps})\right)- I\left(\tPhi^\varsigma_\eps(\tz^-_{\eps})\right)\\
= &\eps\int_{-\infty}^{+\infty} \big[ \{I,H_1\}\left(I,\phi+\nu(I) s+\nu(I)\varsigma,p_0(\tau^*+ s+\varsigma ),q_0(\tau^*+ s+\varsigma)\right)
\\
  &\qquad\qquad -\{I,H_1\}\left(I,\phi+\nu(I) s+\nu(I)\varsigma,0,0\right)\big]\omega(t_0+s+\varsigma) \,  d s
  \\
  &+O(\eps^{1+\rho})\\
  = &\eps M^I(I,\phi+\varsigma,\tau+\varsigma,t_0+\varsigma)+O_{C^1}(\eps^{1+\rho}),
\end{split}\end{equation}

Due to \eqref{eqn:H_1}, $I\left(\tPhi^\varsigma_\eps(\tz^{+}_{\eps})\right)=I\left(\tz^{+}_{\eps}\right)$ and  $I\left(\tPhi^\varsigma_\eps(\tz^{-}_{\eps})\right)=I\left(\tz^{-}_{\eps}\right)$.

By Proposition \ref{prop:Rice2} there exists $\varsigma^*$ such that $t_0+\varsigma^*$ is a point where the process $M^I$ crosses the prescribed value $v$.

Then \eqref{eqn:change_in_I_non_ham_2} implies
\begin{equation}\label{eqn:change_in_I_non_ham_3}\begin{split}
  I\left(\tPhi^\varsigma_\eps(\tz^{+}_{\eps})\right)-I\left(\tPhi^\varsigma_\eps(\tz^-_{\eps})\right)
= &\eps v+O(\eps^{1+\rho}).
\end{split}\end{equation}

This concludes the proof.

\end{proof}

\section{Conclusions and future work}

To summarize, in this paper we considered a rotator-pendulum system with a random perturbation of special type,
and we proved the persistence of the NHIM and of the stable and unstable manifolds, the existence of transverse homoclinic orbits,
and the existence of orbits that exhibit micro-diffusion in the action.
The perturbation  is given by some Hamiltonian vector field that vanishes on the phase space of the rotator, multiplied by unbounded noise.
We work with path-wise  solutions under the assumption  that the sample paths are H\"older continuous.
The persistence of the NHIM is proved not for all times, but only for a distinguished set of times within some arbitrarily large time interval.
It seems possible that the RNHIMs and their stable and unstable manifolds exist for all times,  but may not be of uniform size, or of a  size necessary to guarantee the crossing of the stable and unstable manifolds. Further investigation is planned on this regard,  including also a generalization to less regular sample paths for the noise.

We also plan to show  the existence of Arnold diffusion (rather than micro-diffusion); a foreseeable way is by developing random versions of the lambda lemma and  shadowing lemma, and showing that we can find $O(1/\eps)$ pseudo-orbits of the scattering map that we can join together.

\appendix
\section{Gronwall's inequality}\label{sec:gronwall}

A general form of Gronwall’s lemma is stated in Lemma \ref{lem:gronwall} and it is used to derive the other
three inequalities  that are used in the main text.

\begin{lem}[Gronwall's Lemma]\label{lem:gronwall}
 Let $\alpha$, $\beta$ and $\phi$  be real-valued functions defined on $[t_0,+\infty)$. Assume that $\beta$ and $\phi$ are continuous and that the negative part of $\alpha$ is integrable on every closed and bounded subinterval of $[t_0,+\infty)$.
\begin{itemize}
\item[(i)] If $\beta$ is non-negative and if $\phi$ satisfies the integral inequality
\begin{equation}\label{eqn:ineq1}
  \phi(t)\leq \alpha(t)+\int_{t_0}^{t} \beta(s) \phi(s) ds \textrm { for } t\geq t_0,
\end{equation}
then
\begin{equation}\label{eqn:ineq2}
  \phi(t)\leq \alpha(t)+\int_{t_0}^{t} \alpha(s)\beta(s) \exp\left(\int_{s}^{t} \beta(r) dr \right) ds\textrm { for } t\geq t_0.
\end{equation}

\item[(ii)] If, in addition, the function $\alpha$ is non-decreasing, then
\begin{equation}\label{eqn:ineq3}
  \phi(t)\leq \alpha(t)\exp\left(\int_{t_0}^{t} \beta(s) ds \right) \textrm { for } t\geq t_0.
\end{equation}
\end{itemize}
\end{lem}

For a reference, see e.g. \cite{Pachpatte98}.

\begin{lem}[Gronwall's Inequality -- I]\label{lem:gronwall-I}
Assume that $\delta_0,\delta_1,\delta_2,\delta_3>0$, $t_0\geq 0$, and  $\phi$ is a continuous  function.

If
\[\phi(t)\leq\delta_0+\delta_1 t+\delta_2 t^2+\int_{t_0}^t \delta_3\phi(s) ds, \textrm { for } t\ge t_0,\]
then
\begin{equation}\label{eqn:gronwall-I}
\begin{split}\phi(t)\leq& \left(\delta_0+\delta_1 t+\delta_2 t^2\right) e^{\delta_3 (t-t_0)}\\ <&\left(\delta_0+\delta_1 t+\delta_2 t^2\right) e^{\delta_3 t},
\textrm { for } t\ge t_0.
\end{split}\end{equation}
\end{lem}

\begin{proof}
Let $\alpha(t)=\delta_0+\delta_1 t+\delta_2 t^2$ and $\beta(t)=\delta_3$, where $\delta_0,\delta_1,\delta_2,\delta_3>0$.
Then $\alpha'(t)=\delta_1+2\delta_2 t\geq 0$ for $t\geq 0$, so $\alpha$ is non-decreasing.
Lemma \ref{lem:gronwall}-(ii) implies that
\begin{equation}\label{eqn:ineq4}
  \phi(t)\leq (\delta_0+\delta_1 t+\delta_2 t^2)e^{\delta_3(t-t_0)} \textrm { for } t\geq t_0.
\end{equation}
\end{proof}

\begin{lem}[Gronwall Inequality -- II]\label{lem:gronwall-II}

Assume that  $\delta_0,\delta_1,\delta_2,\delta_3,\delta_4>0$, $t_0\geq 0$, and   $\phi$ is a continuous  function.

If
\[\phi(t)\leq\delta_0+\delta_1 t+\delta_2 t^2+\int_{t_0}^t (\delta_3+\delta_4 s)\phi(s) ds, \textrm { for } t\ge t_0,\]
then
\begin{equation}\label{eqn:gronwall-II}
\begin{split}\phi(t)\leq& \left(\delta_0+\delta_1 t+\delta_2 t^2\right) e^{\left[(\delta_3t+\delta_4\frac{t^2}{2})-(\delta_3t_0+\delta_4\frac{t_0^2}{2})\right]}
\\ <&\left(\delta_0+\delta_1 t+\delta_2 t^2\right) e^{(\delta_3t+\delta_4\frac{t^2}{2})}, \textrm { for } t\ge t_0.
\end{split}\end{equation}
\end{lem}

\begin{proof}
Let $\alpha(t)=\delta_0+\delta_1 t+\delta_2 t^2$ and $\beta(t)=\delta_3+\delta_4 t$, where $\delta_0,\delta_1,\delta_2,\delta_3,\delta_4>0$.
Then
Lemma \ref{lem:gronwall}-(ii) implies that
\begin{equation}\label{eqn:ineq5}
  \phi(t)\leq (\delta_0+\delta_1 t+\delta_2 t^2)e^{\left[(\delta_3 t+\frac{\delta_4}{2}t^2)-(\delta_3 t_0+\frac{\delta_4}{2}t_0^2)\right]} \textrm { for } t\geq t_0.
\end{equation}
\end{proof}


\begin{lem}[Gronwall's Inequality -- III]\label{lem:gronwall-III}
Let $M$ be an $n$-dimensional manifold, $\X^0:M\to TM$ be vector field on $M$ that is Lipschitz in $z\in M$, and $\X^1:M\times\mathbb{R}\times \Omega\to TM$ a time-dependent vector field on $M$ that is Lipschitz in $z\in M$, continuous in $t\in\mathbb{R}$, and measurable in $\omega\in\Omega$.

Consider the following differential equations:
\begin{eqnarray}
\label{eqn:eq_0}\dot{z}(t)&=&\X^0(z),\\
\label{eqn:eq_1}\dot{z}(t)&=&\X^0(z)+\eps \X^1(z, \omega(t_0+t)).
\end{eqnarray}

Assume:
\begin{itemize}
\item $\X^0$  has Lipschitz constant $C_3>0$;
\item For a fixed continuous path $\omega\in\Omega$, $\X^1$ satisfies
\begin{equation}\label{eqn:eq_X_1_omega}\|\X^1(z, \omega(t_0+t))\|\leq C_1t+C_2,
 \end{equation}
  for some $C_1,C_2>0$ depending on $\omega$ and $t_0$ and all $t\geq 0$.
\end{itemize}

Let $z_0$ be a solution of the equation \eqref{eqn:eq_0} and $z_\eps$ be a solution of the equation \eqref{eqn:eq_1} such that
\begin{equation}\label{eqn:3}
\|z_0(0)-z_\eps(0)\|<C_0\eps, \textrm{ for some } C_0>0 \textrm{ depending on $\omega$}.
\end{equation}

Then, for  $0<\varrho_1<1$, $k\leq \frac{1-\rho_1}{C_3}$, there exist  $\eps_0>0$ and  $K$, such that for $0\leq \eps<\eps_0$ we have
\begin{equation}\label{eqn:gronwall-III}
\|z_0(t)-z_\eps(t)\|< K\eps^{\varrho_1}, \textrm{ for } 0\leq t \leq k\ln(1/\eps).
\end{equation}
\end{lem}

\begin{proof}
We have
\begin{equation}\label{eqn:g3-1}\begin{split}
z_0(t)=& z_0(0)+\int_0^t X_0(z_0(s))ds\\
z_\eps(t)=& z_\eps(0)+\int_0^t [ X_0(z_0(s))+\eps X_1(z_\eps(s),\omega (t_0+s))]ds
\end{split}
\end{equation}

\begin{equation}\label{eqn:g3-2}\begin{split}
\|z_\eps(t)-z_0(t)\|\leq &\| z_\eps(0) -z_0(0)\|+\int_0^t \|X_0(z_\eps(s))-X_0(z_0(s))\|ds\\
&+ \eps \int_0^t \|X_1(z_\eps(s),\omega (t_0+s))\| ds\\
\leq& \eps C_0+C_3\int_0^t \| z_\eps(s) - z_0(s)\|ds\\
&+\eps \int_0^t ( C_1s+C_2 ) ds\\
\leq& \eps C_0+\eps C_1t+\eps \frac{C_2}{2}t^2 + C_3\int_0^t \| z_\eps(s) - z_0(s)\|ds\\
\end{split}
\end{equation}

Applying Gronwall Lemma \ref{lem:gronwall-I} for $\delta_0=\eps C_0$, $\delta_1=\eps C_1 $, $\delta_2=\eps \frac{C_2}{2}$ and $\delta_3=C_3$,
we obtain
\begin{equation}\label{eqn:g3-3}\begin{split}
\|z_\eps(t)-z_0(t)\|\leq & \eps \left [ C_0+ C_1t+ \frac{C_2}{2}t^2 \right]e^{C_3 t}
\end{split}
\end{equation}
For $0\leq t\leq k\ln (\frac{1}{\eps})$ we have
\begin{equation}
\label{eqn:g3-4}\begin{split}
\|z_\eps(t)-z_0(t)\|\leq & \eps \left [ C_0+ C_1k\ln (\frac{1}{\eps}) + \frac{C_2}{2}k^2\left(\ln (\frac{1}{\eps})\right)^2 \right]
e^{C_3 k\ln (\frac{1}{\eps})}
\end{split}
\end{equation}

Let $k<\frac{1-\varrho}{C_3}$, where $\varrho\in(0,1)$. Then
\[ e^{C_3 k\ln (\frac{1}{\eps})}\leq \eps ^{-1+\varrho}.\]
From \eqref{eqn:g3-4} we obtain
\begin{equation}
\label{eqn:g3-5}\begin{split}
\|z_\eps(t)-z_0(t)\|\leq & \eps^\varrho \left [ C_0+ C_1k\ln (\frac{1}{\eps}) + \frac{C_2}{2}k^2\left(\ln (\frac{1}{\eps})\right)^2 \right]\\
=&\eps^{\varrho_1} \eps^{\varrho-\varrho_1}\left [ C_0+ C_1k\ln (\frac{1}{\eps}) + \frac{C_2}{2}k^2\left(\ln (\frac{1}{\eps})\right)^2 \right]
\end{split}
\end{equation}
for $0<\varrho_1<\varrho$. Note that $\varrho_1$ can be chosen arbitrary.

There exists $\eps_0$ and constants $A,B>0$ such that, for $0<\eps<\eps_0$, we have
\[\eps^{\varrho-\varrho_1}\ln (\frac{1}{\eps})<A \textrm { and } \eps^{\varrho-\varrho_1}\left (\ln (\frac{1}{\eps})\right)^2<B.\]

Therefore from \eqref{eqn:g3-5} we obtain that, for the  constant \[K=[ C_0+ C_1k A + \frac{C_2}{2}k^2 B ] \] we have
\[ \|z_\eps(t)-z_0(t)\|\leq  K\eps^{\varrho_1}.\]
\end{proof}

\bibliographystyle{alpha}
\bibliography{random}
\end{document}